\RequirePackage{pdf14}
\documentclass[notitlepage,11pt,reqno]{amsart}
\usepackage[foot]{amsaddr}
\usepackage[normalem]{ulem}
\usepackage{amssymb,nicefrac,bm,upgreek,mathtools,verbatim,comment,enumerate}
\usepackage{mathrsfs}
\usepackage{dsfont}
\usepackage[final]{hyperref}
\usepackage{amsopn}

\usepackage[normalem]{ulem}
\newcommand{\stkout}[1]{\ifmmode\text{\sout{\ensuremath{#1}}}\else\sout{#1}\fi}
\usepackage[margin=1in]{geometry}
\allowdisplaybreaks
\hypersetup{
colorlinks=true,
citecolor=mblue,
linkcolor=mblue,
frenchlinks=false,
breaklinks}
\usepackage[msc-links,abbrev,nobysame,backrefs]{amsrefs}
\usepackage{color}
\definecolor{dmagenta}{rgb}{.4,.1,.5}
\definecolor{dblue}{rgb}{.0,.0,.5}
\definecolor{mblue}{rgb}{.0,.0,.7}
\definecolor{ddblue}{rgb}{.0,.0,.4}
\definecolor{dred}{rgb}{.4,.0,.0}
\definecolor{mred}{rgb}{.5,.0,.0}
\definecolor{dgreen}{rgb}{.0,.5,.0}
\definecolor{Eeom}{rgb}{.0,.0,.5}
\definecolor{cm}{cmyk}{1,.0,.0,.0}

\numberwithin{equation}{section}

\theoremstyle{plain}
\newtheorem{theorem}{Theorem}[section]
\newtheorem{lemma}{Lemma}[section]
\newtheorem{corollary}{Corollary}[section]

\theoremstyle{definition}
\newtheorem{assumption}{Assumption}[section]
\newtheorem{definition}{Definition}[section]
\newtheorem{notation}{Notation}[section]

\theoremstyle{remark}
\newtheorem{remark}{Remark}[section]
\newtheorem{example}{Example}[section]

\usepackage[capitalize,nameinlink]{cleveref}

\crefname{section}{Section}{Sections}
\crefname{subsection}{Subsection}{Subsections}
\crefname{condition}{Condition}{Conditions}
\crefname{hypothesis}{Hypothesis}{Conditions}
\crefname{assumption}{Assumption}{Assumptions}
\crefname{notation}{Notation}{Notations}
\crefname{lemma}{Lemma}{Lemmas}
\crefname{fact}{Fact}{Facts}
\crefname{corollary}{Corollary}{Corollaries}
\crefname{theorem}{Theorem}{Theorems}
\crefname{claim}{Claim}{Claims}

\Crefname{figure}{Figure}{Figures}

\crefformat{equation}{\textup{#2(#1)#3}}
\crefrangeformat{equation}{\textup{#3(#1)#4--#5(#2)#6}}
\crefmultiformat{equation}{\textup{#2(#1)#3}}{ and \textup{#2(#1)#3}}
{, \textup{#2(#1)#3}}{, and \textup{#2(#1)#3}}
\crefrangemultiformat{equation}{\textup{#3(#1)#4--#5(#2)#6}}%
{ and \textup{#3(#1)#4--#5(#2)#6}}{, \textup{#3(#1)#4--#5(#2)#6}}%
{, and \textup{#3(#1)#4--#5(#2)#6}}

\Crefformat{equation}{#2Equation~\textup{(#1)}#3}
\Crefrangeformat{equation}{Equations~\textup{#3(#1)#4--#5(#2)#6}}
\Crefmultiformat{equation}{Equations~\textup{#2(#1)#3}}{ and \textup{#2(#1)#3}}
{, \textup{#2(#1)#3}}{, and \textup{#2(#1)#3}}
\Crefrangemultiformat{equation}{Equations~\textup{#3(#1)#4--#5(#2)#6}}%
{ and \textup{#3(#1)#4--#5(#2)#6}}{, \textup{#3(#1)#4--#5(#2)#6}}%
{, and \textup{#3(#1)#4--#5(#2)#6}}

\crefdefaultlabelformat{#2\textup{#1}#3}

\newcommand{\ttup}[1]{\textup{(}#1\textup{)}}
\newcommand{\smid}{\,|\,}
\newcommand{\rcn}{c}
\newcommand{\rcp}{{c_{\mathsf{p}}}}
\newcommand{\pdef}{\bm{M}^{+}}
\newcommand{\psdef}{\bm{M}^{+}_{0}}

\newcommand{\QSp}{\mathbb{Q}} 

\newcommand{\rc}{c}

\newcommand{\Ind}{\mathds{1}}       

\newcommand{\Act}{{\mathds{U}}}
\newcommand{\XX}{{\mathds{X}}}       
\newcommand{\KK}{\mathds{K}}        
\newcommand{\HH}{\mathds{H}}
\newcommand{\cU}{{\mathcal{U}}}     

\newcommand{\uuptau}{{\Breve{\uptau}}}

\newcommand{\Borel}{{\mathscr{B}}}  
\newcommand{\cB}{{\mathcal{B}}}     
\newcommand{\Cc}{{C}}               
\newcommand{\cE}{{\mathcal{E}}}     
\newcommand{\cF}{{\mathcal{F}}}     
\newcommand{\cG}{{\mathcal{G}}}  
\newcommand{\cJ}{{\mathcal{J}}}     
\newcommand{\cK}{{\mathcal{K}}}     
\newcommand{\cL}{{L}}     
\newcommand{\fL}{{\mathfrak{L}}}    
\newcommand{\cM}{{\mathcal{M}}}    
\newcommand{\cN}{{\mathcal{N}}}
\newcommand{\Pm}{{\mathfrak{P}}}    
\newcommand{\cR}{{\mathcal{R}}}     
\newcommand{\cS}{{\mathcal{S}}}     
\newcommand{\cT}{{\mathcal{T}}}     
\newcommand{\cV}{{\mathcal{V}}}     
\newcommand{\eom}{\mathfrak{M}_{\mathsf{erg}}}    
\newcommand{\cX}{{\mathcal{X}}}     

\newcommand{\Uadm}{\mathfrak{U}}

\newcommand{\Usm}{\mathfrak{U}_{\mathrm{sm}}}

\newcommand{\Ustab}{\mathfrak{U}_{\mathsf{stab}}}
\newcommand{\Usms}{\mathfrak{U}^\star_{\mathrm{sm}}}

\newcommand{\RR}{\mathds{R}} 
\newcommand{\NN}{\mathds{N}} 
\newcommand{\ZZ}{\mathds{Z}} 
\newcommand{\Rd}{{\mathds{R}^{d}}}
\DeclareMathOperator{\Exp}{\mathbb{E}} 
\DeclareMathOperator{\Prob}{\mathbb{P}} 
\newcommand{\D}{\mathrm{d}} 

\newcommand{\abs}[1]{\lvert#1\rvert}
\newcommand{\norm}[1]{\lVert#1\rVert}
\newcommand{\babs}[1]{\bigl\lvert#1\bigr\rvert}

\newcommand{\transp}{^{\mathsf{T}}}

\DeclareMathOperator*{\Argmin}{Arg\,min}

\DeclareMathOperator*{\trace}{trace}

\newcommand{\df}{\coloneqq}
\newcommand{\order}{{\mathcal{O}}} 

\makeatletter
\DeclareRobustCommand\widecheck[1]{{\mathpalette\@widecheck{#1}}}
\def\@widecheck#1#2{%
    \setbox\z@\hbox{\m@th$#1#2$}%
    \setbox\tw@\hbox{\m@th$#1%
       \widehat{%
          \vrule\@width\z@\@height\ht\z@
          \vrule\@height\z@\@width\wd\z@}$}%
    \dp\tw@-\ht\z@
    \@tempdima\ht\z@ \advance\@tempdima2\ht\tw@ \divide\@tempdima\thr@@
    \setbox\tw@\hbox{%
       \raise\@tempdima\hbox{\scalebox{1}[-1]{\lower\@tempdima\box
\tw@}}}%
    {\ooalign{\box\tw@ \cr \box\z@}}}
\makeatother

\usepackage{accents}
\newlength{\dhatheight}

\newcommand{\ttl}{\Large
Average cost optimal control under weak ergodicity\\[5pt] hypotheses:
 Relative value iterations}
\begin{document}
\title[Average cost optimal control: Relative value iterations]
{\ttl}

\author[Ari Arapostathis]{Ari Arapostathis$^\dag$}
\address{$^{\dag}$Deceased, was with the Department of ECE,
The University of Texas at Austin,
Austin, TX~~78712}
\email{ari@utexas.edu}

\author[Vivek S. Borkar]{Vivek S. Borkar$^\ddag$}
\address{$^\ddag$Department of Electrical Engineering,
Indian Institute of Technology,
Powai, Mumbai}
\email{borkar@ee.iitb.ac.in}


\begin{abstract}
We study Markov decision processes with Polish state and action spaces.
The action space is state dependent and is not necessarily compact.
We first establish the existence of an optimal ergodic
occupation measure using only a near-monotone hypothesis on the running cost.
Then we study the well-posedness of Bellman equation,
or what is commonly known as the average cost optimality equation, under
the additional hypothesis of the existence of a small set.
We deviate from the usual approach which is based on the vanishing discount
method and instead map  the problem to an equivalent one for a controlled split chain.
We employ a stochastic representation of the Poisson equation to derive the
Bellman equation.
Next, under suitable assumptions,
we establish convergence results for the `relative value iteration'
algorithm which computes the solution of the Bellman equation recursively.
In addition, we present some results
concerning the stability and asymptotic optimality
of the associated rolling horizon policies.
\end{abstract}

\keywords{ergodic control, Bellman equation, inf-compact cost, relative value iteration}
\subjclass[2000]{Primary: 90C40, Secondary: 93E20}
\maketitle

\section{Introduction}

The long run average or `ergodic' cost is popular in applications when transients
are fast and/or unimportant and one is optimizing over possible asymptotic behaviors.
The dynamic programming equation for this criterion, in the finite state-action case,
goes back
to Howard \cite{Howard-60}. A recursive algorithm to solve it in the aforementioned
case is the so called relative value iteration scheme \cite{White-63},
dubbed so because it is a modification of the value iteration scheme for
the (simpler) discounted cost criterion. This  modification consists of subtracting
at each step a suitable offset and track only the `relative' values.
Suitable counterparts of this algorithm for a general state space are available,
if at all, only under rather strong conditions (see, e.g., Section~5.6 of
\cite{ref:HLL1}). Our aim here is to consider a special case of immense
practical importance, viz., that of a near-monotone or inf-compact cost which penalizes
instability \cite{Borkar-91}, \cite{Borkar-02}, and to establish both the dynamic programming equation
and the relative value iteration scheme for it.
Perforce the latter involves iteration in a function space and
as far as implementation is concerned,
would have
to be replaced by suitable finite approximations through either state
aggregation or parametrized approximation of the value function.
But the validity of such an approximate scheme
depends on provable convergence properties
of the algorithm. Our aim is to provide this.

The results on convergence of the relative value iteration presented
here may be viewed as discrete time counterparts of the
results of \cite{RVIM}. It is not, however, the case that they can be derived
simply from the results of \cite{RVIM}, which relies heavily on the analytic
machinery of the partial differential equations arising therein.
This, in
particular, leads to convenient regularity results which are not available here.

For studies on the \emph{average cost optimality equation} (ACOE)
 of Markov decision processes
(MDP) on Borel state space, we refer the reader to
\cites{Costa-12,Feinberg-02,Feinberg-12,Feinberg-17,H-L-91,Jaskiewicz-06,Schal,
Vega-Amaya-03,Vega-Amaya-18,ref:HLL1}.
All these papers assume only the (weak) Feller property on the transition kernel,
whereas in this paper the kernel is assumed to be strong Feller
(with the exception of \cref{L2.1,T2.1}).
Classical approaches based on the vanishing discount argument such as
\cites{Costa-12,Schal,H-L-91,Feinberg-12}
need to ensure some variant of pointwise boundedness
of the relative discounted value function.
This typically requires additional hypotheses, or it is directly
imposed as an assumption.
The  weakest condition appears in \cite{Feinberg-12} where only the limit
infimum of relative discounted value functions is required to be pointwise
bounded in the vanishing discount limit.
It follows from \cite[Theorem~4.1]{FAM-90} that if the solution of
the ACOE is bounded then the relative discounted value functions are also bounded
uniformly in the discount factor.
This is a very specific case though,
and for the more general case studied in this paper
it is unclear how our assumptions compare with those of \cite{Feinberg-12}.

Pointwise boundedness of discounted relative value functions was verified
from scratch for a specific application in \cite{Agarwal-08}.
The techniques therein, which leverage near-monotonicity in a manner
different from here, may be more generally applicable.

 Some of the aforementioned  works derive an
\emph{average cost optimality inequality} (ACOI) as opposed to an equation.
The ACOE is derived in \cites{Costa-12,Feinberg-02,ref:HLL1,Vega-Amaya-03,Feinberg-17,
Jaskiewicz-06,Vega-Amaya-18}.
Also worth noting is \cite{Feinberg-17} which derives the ACOE for a
classical inventory problem under a weak condition known as $\cK$-inf-compactness.
The works in \cites{Jaskiewicz-06,Vega-Amaya-18}
derive the ACOE under additional uniform stability conditions which we avoid. The works \cite{Vega-Amaya-03}, \cite{Vega-Amaya-18} also use a minorization condition like us, but the purpose there is to facilitate a fixed point argument which is possible due to the stronger stability assumptions. 

Our focus is on the ACOE rather than the ACOI because our eventual aim
is to establish convergence of relative value iteration for which this is
explicitly used.
Moreover, for this convergence result we require uniqueness of the
solution to the ACOE within a suitable class of functions.

Studies such as \cites{Schal,Feinberg-12,H-L-91} work with standard Borel
state spaces whereas we work with  Polish spaces.
We assume that the running cost is near-monotone
(see \hyperlink{C}{(C)} in \cref{S2.2}), a notion more general than the more
commonly used `inf-compactness'.
The latter requires the level sets of the
running cost functions to be compact, necessitating in particular that for
non-$\sigma$-compact spaces, they be extended real-valued.
On the contrary, a $\cK$-inf-compact cost
(see \hyperlink{A1}{(A1)} in \cref{S3.1})
together with \hyperlink{C}{(C)}
allows for more flexibility.

Furthermore, the above works do not address the relative value
iteration which is our main focus here.
This algorithm, after the seminal work of \cite{White-63} for the finite state case,
has been extended to denumerable state spaces in
\cites{Cavazos-96a,Cavazos-98,ChenMeyn-99}.
An analogous treatment for a general metric state space appears in
\cite[Section~5.6]{ref:HLL1}.
This directly assumes equicontinuity of the iterates, for which problems with
convex value functions \cite{FAM-92} have been cited as an example.
We do not make any such assumption. The related though distinct algorithm of policy 
iteration has been analyzed in \cite{Meyn-97}. This work also uses the `pseudo-atom' construction as we do, in order  to obtain a solution to the fixed policy Poisson equation. We use it to derive the Bellman equation itself using a representation of the value function.

Another important part of this work concerns the stability
and asymptotic optimality of rolling horizon policies.
Analogous results in the literature have been reported only under very strong
blanket ergodicity assumptions \cites{HLL-90,Cavazos-98}.
For a review of this topic, see \cite{Vecchia-12}.
In this paper, we avoid any blanket ergodicity  assumptions
and impose a stabilizability hypothesis, namely, that under some Markov control
the process is geometrically ergodic with a Lyapunov function that has the
same growth as the running cost (see \hyperlink{H2}{(H2)} and \cref{R6.1} in \cref{S6}).
This property is natural for `linear-like' problems, and is also
manifested in queuing problems with abandonment, or problems with
the structure in \cref{Ex6.1}.
Under this hypothesis, we assert in \cref{T6.1,T6.3}, global convergence for the relative value
iteration, and show in \cref{T6.2} that the rolling horizon procedure
is stabilizing after a finite number of iterations.
Then, under a `uniform' $\psi$-irreducibility condition,
\cref{T6.4} shows that the rolling horizon procedure is asymptotically optimal.
The latter is an important problem of current interest
(see, e.g., \cites{ref:HLL1,Chatterjee-15}).
Our results also contain computable error bounds.

The article is organized as follows.
\Cref{S2} has three main parts.
We first review the  formalism and basic notation of Markov decision processes in \cref{S2.1}
and then, in \cref{S2.2}, we establish the existence of an optimal ergodic occupation measure,
thus extending the results of the convex analytic framework in \cite{Borkar-02}
to MDPs on a Polish state space.
\Cref{S2.4} introduces an equivalent controlled split chain and the associated
\emph{pseudo-atom}.
\cref{S3} then derives the dynamic programming equation to characterize optimality, extending the approach of \cite{Borkar-91} for countable state space - compact action space case.
\cref{S4} establishes the convergence of the `value iteration', which is the name
we give to the analog of value iteration for discounted cost with no discounting,
but with the cost-per-stage function modified by subtracting from it the optimal cost.
The latter is in principle unknown, so this is not a legitimate algorithm.
It does, however, pave the way to prove convergence of the true relative value
iteration scheme, which we do in \cref{S5}.
\cref{S6} is devoted to the analysis of the rolling horizon procedure.

\subsection{Notation}\label{S-not}
We summarize some notation used throughout the paper.
We use $\Rd$ (and $\mathbb{R}^d_+$), $d\ge 1$, to denote the space of real-valued
$d$-dimensional (nonnegative) vectors, and write $\RR$ for $d=1$.
Also, $\NN$ denotes the natural numbers, and $\NN_0\df\NN\cup\{0\}$.
For $x, y\in \RR$, we let
$$x \vee y \,\df\, \max\{x,y\}\quad\text{and\ \ }x\wedge y \,\df\, \min\{x,y\}\,.$$
The Euclidean norm on $\Rd$ is denoted by $|\cdot|$.
We use $A^{c}$, $\Bar{A}$, and $\Ind_{A}$ to denote
the complement, the closure, and the indicator function of a set $A$, respectively.

For a Polish space $\XX$ we let $\Borel(\XX)$ stand for its Borel
$\sigma$-algebra, and $\Pm(\XX)$ for the space of probability measures
on $\Borel(\XX)$ with the Prokhorov topology.
We let $\cM(\XX)$, $\fL(\XX)$, and $C(\XX)$,
denote the spaces of real-valued Borel measurable functions,
lower semi-continuous functions bounded from below, and continuous functions
on $\XX$, respectively.
Also, $\cM_b(\XX)$, $\fL_b(\XX)$, and $C_b(\XX)$, denote
the corresponding subspaces consisting of bounded functions.

For a Borel probability measure $\mu$
on $\Borel(\XX)$ and a measurable function $f\colon\XX\to\RR$,
which is integrable under $\mu$, we often use the convenient notation
$\mu(f)\df\int_\XX f(x)\,\mu(\D{x})$.

For $f\in\cM(\XX)$, we define
\begin{equation*}
\norm{g}_f \,\df\, \sup_{x\in\XX}\, \frac{\abs{g(x)}}{1+\abs{f(x)}}\,,\quad
g\in\cM(\XX)\,,
\end{equation*}
and $\order(f) \df \{g\in\cM(\XX)\colon \norm{g}_f<\infty\}$.

\subsection{Assumptions} In this subsection, we outline the various assumptions used in this article. These have been introduced closer to their first use and after the relevant notation is in place. Not all of them are required for everything.

\Cref{A2.1} introduced in \Cref{S2.1} is the basic assumption regarding the minimal regularity hypothesis about the transition kernel, the set-valued map specifying available controls at each state, and the cost function. This is assumed throughout this work. 

Assumption \hyperlink{C}{(C)} in  \Cref{S2.2} refines further the assumption on the cost function. This too holds throughout and is first used in  \Cref{T2.1}.

Assumption  \hyperlink{A0}{(A0)} in  \Cref{S2.4} is an adaptation of the standard `minorization' condition for the construction of the Athreya-Ney-Nummelin pseudo-atom for our purposes and facilitates the derivation of the Poisson equation for the split chain in this section.  \hyperlink{A1}{(A1)} and \hyperlink{A2}{(A2)} in  \Cref{S3.1} are additional assumptions  for the passage  from the Poisson equation to the \textit{Bellman} equation in \Cref{T3.1}.

Assumption \hyperlink{}{A3.1} in \Cref{S3.3} strengthens our regularity requirement on the controlled transition kernel, from weak feller to strong Feller. This is required in the derivation of the Bellman equation and is therefore operative throughout the rest of the article. It also plays a role in the subsequent analysis of the relative value iteration algorithm.

One of our main results is the convergence of the relative value iteration to solve the Bellman equation. This requires the additional assumption \hyperlink{H1}{(H1)} in  \Cref{S4.2}, which proves the convergence of value iteration (\Cref{L4.1}, \Cref{T4.1}) assuming the optimal cost to be known. Convergence of relative value iteration in \Cref{T5.1} of \Cref{S5} follows from this. (Part (c) of the theorem invokes  \hyperlink{H2}{(H2)} from the subsequent section, but parts (a), (b) do not require it.) 

Assumption  \hyperlink{H2}{(H2)} and its equivalent statement \hyperlink{H2'}{(H2')} are used to justify the rolling horizon procedure in \Cref{T6.1}, \Cref{T6.2}, which requires stronger conditions. The corresponding convergence result for optimal policies  in \Cref{T6.3} needs additional conditions that are embedded in the statement of the theorem itself.

\section{Preliminaries}\label{S2}

In this paper, we consider a controlled Markov chain
otherwise referred to as a Markov decision process (MDP),
taking values in a Polish space $\XX$.

\subsection{The MDP model}\label{S2.1}
Recall the notation introduced in \cref{S-not}.
According to the most prevalent definition in the literature
(see \cites{ref:HLL1,Dyn-Yush}), an MDP is represented 
as a tuple $\bigl(\XX,\Act,\cU,P,\rc\bigr)$,
whose elements can be described as follows.
\begin{itemize}
\item[(a)]
The \emph{state space} $\XX$ is a Polish space (complete, separable, metric).
Its elements are called \emph{states}.
\item[(b)]
$\Act$ is a Polish space, referred to as the \emph{action} or
\emph{control space}.
\item[(c)]
The map $\cU\colon \XX\to \Borel(\Act)$ is a strict, measurable multifunction.
The set of admissible state/action pairs is defined as
\begin{equation*}
\KK \,\df\, \bigl\{(x,u)\colon\, x\in\XX,\,u\in\cU(x)\bigr\}\,,
\end{equation*}
endowed with the relative topology
corresponding to $\XX\times\Act$.
\item[(d)]
The transition probability $P(\cdot\,|\, x,u)$ is a stochastic kernel on
$\KK\times\Borel(\XX)$, that is, $P(\,\cdot\,|\, x,u)$ is
a probability measure on $\Borel(\XX)$ for each $(x,u)\in\KK$,
and $(x,u) \mapsto P(A\,|\, x,u)$ is in $\cM(\KK)$ for each $A\in\Borel(\XX)$.
\item[(e)]
The map $c\colon \KK\to\RR$ is measurable, and is called the \emph{running
cost} or \emph{one stage cost}.
We assume that it is bounded from below in $\KK$, so without loss of generality,
it takes values in $[1,\infty]$.
\end{itemize}

The (admissible) \emph{history spaces} are defined as
\begin{equation*}
\HH_0\,\df\, \XX\,, \quad \HH_{t} \,\df\, \KK^{t-1}\times\XX\,,\quad t\in\NN\,,
\end{equation*}
and the canonical sample space is defined as $\Omega\df (\XX\times\Act)^\infty$.
These spaces are endowed with their respective product topologies and
are therefore Polish spaces.
The state, action (or control), and information processes, denoted
by $\{X_t\}_{t\in\NN_0}$,  $\{U_t\}_{t\in\NN_0}$ and $\{H_t\}_{t\in\NN_0}$,
respectively, are defined by the projections
\begin{equation*}
X_t(\omega) \,\df\, x_t\,,\quad U_t(\omega) \,\df\, u_t\,,
\quad H_t(\omega) \,\df\, (x_0,\dotsc, u_{t-1}, x_t)
\end{equation*}
for each $\omega=(x_0,\dotsc,u_{t-1},x_t, u_t,\dotsc)\in\Omega$.

An \emph{admissible control strategy}, or \emph{policy}, is a sequence
$\xi = \{\xi_t\}_{t\in\NN_0}$ of stochastic kernels on
$\HH_t\times\Borel(\Act)$ satisfying the constraint
\begin{equation*}
\xi_t(\cU(x_t)\,|\, h_t) \,=\, 1\,,\quad x_t\in\XX\,,\; h_t\in\HH_t\,.
\end{equation*}
The set of all admissible strategies is denoted by $\Uadm$.
It is well known
(see \cite[Prop.\ V.1.1, pp.~162--164]{Neveu})
that for any given $\mu\in\Pm(\XX)$ and $\xi\in\Uadm$ there exists
a unique probability measure $\Prob^\xi_\mu$ on $\bigl(\Omega,\Borel(\Omega)\bigr)$
satisfying
\begin{align*}
\Prob^\xi_\mu(X_0\in D) &\,=\, \mu(D)\qquad\forall\, D\in\Borel(\XX)\,,\\
\Prob^\xi_\mu(U_t\in C\,|\, H_t) &\,=\, 
\xi_t(C\,|\, H_t)
\quad \Prob^\xi_\mu\text{-a.s.}\,,\quad \forall\, C\in\Borel(\Act)\,,\\
\Prob^\xi_\mu(X_{t+1}\in D\,|\, H_t,U_t)  &\,=\,
P(D\,|\, X_t, U_t) \quad \Prob^\xi_\mu\text{-a.s.}\,,
\quad\forall\, D\in\Borel(\XX)\,.
\end{align*}
The expectation operator corresponding to $\Prob^\xi_\mu$
is denoted by $\Exp^\xi_\mu$.
If $\mu$ is a Dirac mass at $x\in\XX$, we simply write these
as $\Prob^\xi_x$ and $\Exp^\xi_x$.

A strategy $\xi$ is called \emph{randomized Markov}
if there exists a sequence of measurable maps $\{v_t\}_{t\in\NN_0}$,
where $v_t\colon \XX\to \Pm(\Act)$ for each $t\in\NN_0$,
such that
$$\xi_t(\,\cdot\,|\, H_t) \,=\, v_t(X_t)(\cdot) \quad \Prob^\xi_\mu\text{-a.s.}$$
With some abuse of notation, such a strategy is identified
with the sequence $v=\{v_t\}_{t\in\NN_0}$.
Note then that
 $v_t$ may be written as a stochastic kernel $v_t(\cdot\,|\, x)$ on
$\XX\times\Borel(\Act)$ which satisfies $v_t(\cU(x)\,|\, x)=1$.

We say that a  Markov randomized strategy $\xi$ is \emph{simple}, or \emph{precise},
if $\xi_t$ is a Dirac mass, in which case $v_t$ is identified
with a Borel measurable function $v_t\colon \XX\to\Act$.
In other words, $v_t$ is a measurable selector from the set-valued map $\cU(x)$ \cite{Feinberg-13}.

We add the adjective \emph{stationary} to indicate that the strategy
does not depend on $t\in\NN_0$, that is, $v_t=v$ for all $t\in\NN_0$.
We let $\Usm$ denote the class of stationary Markov randomized strategies,
henceforth referred to simply as \emph{stationary strategies}.

The basic structural hypotheses on the model, which are assumed
throughout the paper, are as follows.

\begin{assumption}\label{A2.1}
The following hold:
\begin{itemize}
\item[(i)]
The transition probability $P(\D{y}\,|\, x,u)$ is \emph{weakly continuous},
that is, the map $$(x,u)\,\mapsto\, H_f(x,u) := \int_\XX f(y)P(\D{y}\,|\, x,u)$$
is continuous for every $f\in C_b(\XX)$.
\item[(ii)]
The set-valued map $\cU\colon\XX\to\Borel(\Act)$ is upper semi-continuous
and closed-valued.
\item[(iii)]
The running cost $\rc\colon\KK\to[1,\infty]$ is lower semi-continuous.
\end{itemize}
\end{assumption}

\Cref{A2.1} is assumed throughout the paper, and repeated only for emphasis.
More specific assumptions are imposed later in \cref{S2.4} and \cref{S3}.

\begin{definition}\label{D2.1}
For $v\in\Usm$ we use the abbreviated notation
\begin{equation*}
P_v(A\,|\,x) \,\df\, \int_{\cU(x)} P(A\,|\,x,u)\,v(\D{u}\,|\,x)\,,
\quad\text{and\ \ }
\rc_v(x) \,\df\, \int_{\cU(x)}\rc(x,u)\,v(\D{u}\,|\,x)\,.
\end{equation*}
Also,
$P_v f(x) \df \int_\XX f(y) P_v(\D{y}\,|\,x)$ for a function $f\in\cM(\XX)$,
assuming that the integral is well defined.
Similarly, we write $P_u(A\,|\,x) \df P\bigl(A\,|\,x,u\bigr)$ for $u\in\cU(x)$,
and define $P_u f$ analogously.
When needed to avoid ambiguity, we denote the chain controlled
under $v$ as $\{X_n^v\}_{n\in\NN_0}$.
\end{definition}

\subsubsection{Control objective}\label{S2.1.1}

The control objective is to minimize over all admissible $\xi=\{\xi_n\}_{n\in\NN_0}$ the
average (or `\textit{ergodic}') cost
\begin{equation*}
\cE\bigl(\mu,\xi\bigr)\,\df\,\limsup_{N\to\infty}\,\frac{1}{N}\,
\Exp_\mu^\xi\Biggl[\sum_{n=0}^{N-1}\rc(X_n, U_n)\Biggr]\,,\quad\mu\in\Pm(\XX)\,,\ 
\xi\in\Uadm\,.
\end{equation*}
We let
\begin{equation}\label{E-betax}
J(\mu) \,\df\, \inf_{\xi\in\Uadm}\;\cE\bigl(\mu,\xi\bigr)\,,
\qquad\text{and\ \ } \beta\,\df\, \inf_{\mu\in\Pm(\XX)}\,J(\mu)\,.
\end{equation}
We say that an admissible strategy $\xi$
is \emph{optimal} if $\cE\bigl(\mu,\xi\bigr)=J(\mu)$ for all
$\mu\in\Pm(\XX)$.
The class of Markov stationary strategies that are optimal is
denoted by $\Usms$.

In the next section we introduce the concept of an
\emph{optimal ergodic occupation measure}, and assume that, under a near-monotone
type hypothesis on the running cost, such a measure exists.
We use this result in \cref{S3} to derive a solution to the Bellman equation.

\subsection{Existence of an optimal ergodic occupation measure}\label{S2.2}

Recall that $\zeta\in\Pm(\KK)$ is called an
\emph{ergodic occupation measure} if it satisfies
\begin{equation*}
\int_{\KK} \biggl(f(x) - \int_{\XX} f(y) P(\D{y}\,|\,x,u)\biggr)
\zeta(\D{x},\D{u}) \,=\, 0\quad\forall\,f\in\Cc_b(\XX)\,.
\end{equation*}
We let $\eom$ stand for the class of ergodic occupation measures.
Any $\zeta\in\eom$ can be disintegrated as
\begin{equation}\label{E-disint}
\zeta(\D{x},\D{u}) \,=\, \uppi_\zeta(\D{x})\, v_\zeta(\D{u}\,|\,x)
\quad \uppi_\zeta\text{-a.s.}\,,
\end{equation}
where $\uppi_\zeta\in\Pm(\XX)$ and $v_\zeta$ is
a stochastic kernel on
$\XX\times\Borel(\Act)$ which satisfies $v_\zeta(\cU(x)\,|\, x)=1$.
We denote this disintegration as $\zeta=\uppi_\zeta\circledast v_\zeta$.

\begin{remark}
Note that \cref{E-disint} does not define $v_\zeta$ on the entire space,
and thus $v_\zeta$ cannot be viewed as an element of $\Usm$.
However, if $v\in\Usm$ is any strategy that agrees
$\uppi_\zeta$-a.e.\ with $v_\zeta$, then
$\uppi_\zeta(\,\cdot\,) = \int_\XX P_v (\,\cdot\,|\,x)\,\uppi_\zeta(\D{x})$,
or, in other words, $\uppi_\zeta$ is an invariant probability measure for
the chain controlled under $v$.
Note that such a strategy can be easily constructed.
For example, for arbitrary $v_0\in\Usm$, we can define
$v=v_\zeta$ on the support of $\uppi_\zeta$ and $v=v_0$ on its complement.
\end{remark}

\begin{definition}\label{D2.2}
We say that $\zeta^\star\in\eom$ is optimal if
\begin{equation}\label{ED2.2A}
\int_\KK \rc\,\D\zeta^\star\,=\, \beta
\end{equation}
and denote the set of optimal ergodic occupation measures by $\eom^\star$.
\end{definition}

The convex analytic method introduced in \cite{Borkar-88} (see also \cite{Borkar-02})
is a powerful tool for the analysis of ergodic occupation measures.
Two main models have been considered: MDPs with a blanket stability property,
and MDPs with a \emph{near-monotone} running cost. Near-monotonicity is a structural assumption,
which, stated in simple terms, postulates that the running cost is strictly larger
than the optimal average value in \cref{E-betax} on the complement of some compact set.
More precisely, this assumption is stated as follows:

\smallskip
\begin{itemize}
\item[\hypertarget{C}{\textbf{(C)}}]
Consider a continuous one-one embedding $\Psi: \KK \to \KK^*$ of $\KK$ into  a Polish space $\KK^*$ such that $\overline{\Psi(\KK)}$ is compact in $\KK^*$. (Existence and examples of such embeddings follow.) By abuse of notation,  we identify $\KK$ with its image $\Psi(K)$  under this map and $\KK^*$ with $\overline{\Psi(\KK)}$. 
Furthermore, we assume that there exists a compact set $\widetilde{K}\subset\KK$ and a $\varepsilon_0 > 0$ such that 
\begin{equation}\label{E-C}
\rc(x,u) \,\ge\, \beta + \varepsilon_0
\qquad\forall (x,u)\in\KK\setminus
\widetilde{K}\,.
\end{equation}
This implies in particular that $\{(x_n,u_n)\} \subset \KK, (x_n,u_n) \to \partial \KK := \KK^*\backslash\KK$, then
\begin{equation}\label{nmonotone}
\liminf_{n\uparrow\infty}c(x_n,u_n) \ge \beta + \varepsilon_0.
\end{equation}
It will be convenient for us to take $\KK^*$ to be the closure of $\KK$ $(\approx \Psi(\KK))$ embedded in $\XX^*\times\Act^*$ where $\XX^*,\Act^*$ are resp., compact dense embeddings of $\XX, \Act$ into suitable Polish spaces, assumed to exist. We shall assume that this is so.
\end{itemize}

For MDPs on a countable state space a natural counterpart of assumption \hyperlink{C}{(C)}
 is enough to guarantee the existence
of an optimal ergodic occupation measure
as shown in \cite{Borkar-02}.
In \cref{T2.1}, we extend this result to MDPs on a Polish space
under \cref{A2.1} and the above assumption.

As an example, consider the case where the state space $\XX$ is locally compact
 and $\KK=\XX\times\Act$ for a compact action space $\Act$.
Suppose that a sequence $\{\widetilde\zeta_n\}_{n\in\NN}$ of
mean empirical measures converges vaguely to a positive measure $\mu\in\Pm(\KK)$,
meaning that $\int_\KK f\D\widetilde\zeta_n\to\int_\KK f\D\mu$
as $n\to\infty$ for all $f\in C_c(\KK)$, where
$C_c(\KK)$ denotes the subspace of $C_b(\XX)$ consisting
of functions with compact support.
A key lemma then asserts that $\mu(\KK) > 0$,
the normalized measure $\frac{\mu}{\mu(\KK)}$ on $\KK$ is an ergodic occupation measure.
This is established in \cite[Lemma~2.6]{Borkar-02} for models with a countable
state space, and the proof can be adapted to MDPs with
a locally compact state space.
An important ingredient in this proof is employing the Alexandroff extension,
commonly known as the one-point compactification,
and then applying Prokhorov's theorem to
the compactified space $\XX\cup\{\infty\}$.

The Alexandroff extension has a simple and geometrically meaningful structure,
but it does not result in a Hausdorff compactification unless the original space
is locally compact.
For models with general Polish state and action spaces, a general scheme that is always available is to employ Urysohn's theorem
to embed $\KK$ in the Hilbert cube, and use the closure of its image as
a compactification.
This is done as follows.

\begin{definition}[Embedding in the Hilbert cube]\label{D2.4}
As is well known,
$\KK$, being a Polish space, can be homeomorphically embedded
as a $G_{\delta}$ subset of the Hilbert cube $[0,1]^\infty$
by a homeomorphism $\Psi \colon\KK \leftrightarrow \Psi(\KK) \subset [0,1]^\infty$
\cite[Propositions~7.2 and 7.3]{BertsShreve}.
Thus we can identify  $\KK$ with $\Psi(\KK)$.
Let $\KK^* \df \overline{\Psi(\KK)}$, and view $\KK$ as being densely
homeomorphically embedded in $\KK^*$ with
$\partial\KK \df \KK^*\setminus\KK$.
We may view $\Pm(\KK)$ as being isometrically embedded in $\Pm(\KK^*)$ in the obvious manner.
The latter is compact by Prokhorov's theorem.
\end{definition}

 This may not always be convenient and one may use better problem-specific choices. As an example, consider $\KK :=$ a closed bounded subset of  $C_1[0,1]$, the space of continuous real-valued functions on $[0,1]$ which are continuously differentiable on  $(0,1)$ with left, resp.\ right limits at $0,1$, equipped with the norm $\|f\|_1 := \sup_{x\in [0,1]}|f(x)| + \sup_{x\in (0,1)}|f'(x)|$.  Its natural embedding into $C[0,1] :=$  the space of continuous real-valued functions on $[0,1]$ with the sup-norm, is compact and dense. Thus taking any bounded subset of $C_1[0,1]$ as state space with extended real valued cost $c(x,u) := \|f\|_* := \|f\|_1  + \sup_{x\in (0,1)}|f''(x)|$, satisfies the above conditions. Further examples can be constructed using compact embedding theorems for Sobolev and H$\ddot{o}$lder spaces such as the ones provided by the Rellich-Kondrachov theorem. Another example is a bounded subset of the space of probability measures on a euclidean space with finite $p$-th moment, $p \geq 1$, with the Wasserstein-$p$ distance, embedded in the space of all probability measures on the underlying space with Prohorov topology. A set of laws with uniformly bounded $p$-th moment is necessarily tight, hence the embedding is compact. Density follows easily. Yet another simple example is the natural embedding of the open unit ball $\{f : \|f\| < 1\}$ in $L_2[0,1]$ with norm topology with its natural embedding into $L_2[0,1]$ with the weak$^*$ topology.

Let $\tilde{F}$ to be the class of  functions
in $C_b(\XX^*)$.
The functions in $\tilde{F}$ can also be viewed as functions on $\KK^*$ by letting
$f(x,u)\equiv f(x)$ for $u\in\cU^*(x)$.
Abusing the notation, we use the same symbol $\tilde{F}$ to denote the pullback of the family
$\tilde{F}$ by the map $\Psi^{-1}$.
These are functions on $\Psi(\XX)$, that is,
$f(z)$ for $z\in\Psi(\XX)$ is identified with $f\bigl(\Psi^{-1}(z)\bigr)$.
Since $f(x,u)$ in the family $\tilde{F}\subset C_b(\KK^*)$
does not depend on $u$, abusing the notation,
we denote it simply as $f(x)$ whenever this is convenient. 


In the study of the average cost problem, empirical occupation measures
play an important role.
These are defined as follows.

\begin{definition}
For any given $\mu_\circ\in\Pm(\XX)$
and $\xi\in\Uadm$, we define the family of \emph{mean empirical measures}
$\bigl\{\widetilde\zeta_t \in \Pm(\KK)\,,\;t>0\bigr\}$ by:
\begin{equation*}
\int_{\KK}h(x,u)\,\widetilde\zeta_t(\D{x},\D{u}) \,\df\, \frac{1}{t}
\sum_{m=0}^{t-1}
\Exp_{\mu_\circ}^\xi\bigl[h(X_m,\xi_m)\bigr]\quad \forall \, h \in C_b(\KK)\,.
\end{equation*}
Naturally, $\widetilde\zeta_t$
 depends on $\mu_\circ$ and $\xi$, but we suppress this dependence in the notation.
\end{definition}


We state and prove a key lemma which is analogous to the one mentioned earlier
for the locally compact case. Consider a sequence $\bigl\{\widetilde\zeta_n\bigr\}_{n\in\NN}$ of mean empirical measures viewed as a sequence in $\Pm(\KK^*)$ using the embedding in Definition 2.3. By Prokhorov's theorem, moving to a subsequence if necessary,
also denoted as $\{\widetilde\zeta_n\}_{n\in\NN}$,
we have $\widetilde\zeta_n\Rightarrow\widehat\zeta$
for some $\widehat\zeta\in\Pm(\KK^*)$. Since $\KK^*$ is the disjoint union of $\KK$ and $\partial \KK$,
it is clear that $\widehat\zeta$ must be of the form
\begin{equation}\label{PT2.1C}
\widehat\zeta \,=\, a\zeta_0 + (1 - a)\zeta_1
\end{equation}
for some $a \in [0,1]$, $\zeta_0\in\Pm(\partial\KK)$,
and $\zeta_1\in \Pm(\KK)$.

\begin{lemma}\label{L2.1}
If $a < 1$, then $\zeta_1 \in \eom$. The same conclusion applies
for a sequence $\{\zeta_n\}_{n\in\NN}\subset\eom$.
\end{lemma}

\begin{proof}
Using the strong law of large numbers for martingales given by
\begin{equation*}
\frac{1}{t}\sum_{m=1}^t
\Bigl(f(X_m) - \Exp_{\mu_\circ}^\xi
\bigl[f(X_m)\,|\,X_{m-1},U_{m-1}\bigr]\Bigr)
\,\xrightarrow[t\to\infty]{}\, 0 \ \ \text{a.s.}
\end{equation*}
for $f\in C_b(\XX)$,
we obtain, upon taking expectations, that
\begin{equation}\label{PT2.1B}
\int_{\KK} \biggl(f(x) - \int_\XX f(y) P(\D{y}\,|\,x,u)\biggr)
\widetilde\zeta_t(\D{x},\D{u}) \,\xrightarrow[t\to\infty]{}\, 0\,.
\end{equation}

Then
\begin{equation}\label{PT2.1Cc}
\begin{aligned}
\lim_{n\to\infty}\,\int_{\KK} f\,\D\widetilde\zeta_n &\,=\, \lim_{n\to\infty}\,\int_{\KK^*} f\,\D\widetilde\zeta_n \,=\,
\int_{\KK^*} f\,\D\widehat\zeta \\[3pt]
&\,=\,
a \int_{\partial\KK} f\,\D\zeta_0 + (1-a) \int_{\KK} f\,\D\zeta_1\qquad
\forall f\in C_b(\KK^*)
\end{aligned}
\end{equation}
by the hypothesis that $\widetilde\zeta_n\Rightarrow\widehat\zeta$.
As shown in \cite[Theorem~4.5]{EthKurtz}, if $\mathcal{X}$ is Polish,
then any subset $\bm{F}\subset C_b(\mathcal{X})$  which
\emph{separates points} in $\mathcal{X}$
and is also an algebra is a separating class for Borel probability measures,
meaning that if $\mu',\mu''\in\Pm(\mathcal{X})$ satisfy
$\int f\D\mu'=\int f\D\mu''$ for all $f\in\bm{F}$ then
$\mu'=\mu''$.
The method that we use in this proof reduces the problem
of proving that $\zeta_1\in\eom$  to establishing
equality of two given measures in $\Pm(\XX)$.
Therefore, it suffices to continue with a class $\bm{F}$ that only separates
probability measures. By adding a constant to any $f\in\bm{F}$, we may suppose that each $f\in\bm{F}$ is bounded away from zero from below.  We begin with a special subclass of such $f$. Recall that given a compatible metric $d: \KK\times\KK \mapsto [0,1]$, and a countable dense set $\{s_n\}$ in $\KK$, we can homeomorhically embed $\KK$ into $[0,1]^\infty$ via the map $\Phi: s \in \KK \mapsto [(d(s,s_1), d(s,s_2), \cdots] \in \Phi(\KK) \subset [0,1]^\infty$ (See Definition \ref{D2.4}). Then for any $n\geq 1$, the set $\KK_n := \{[(d(s,s_1), d(s,s_2), \cdots , d(s,s_n)]\}$ is locally compact in the relative topology of $[0,1]^n$. Let $\KK_n^* = \KK_n\cup\{\infty\}$ denote its one point compactification. Consider $f$ above of the form $f(s) = g(d(s,s_1) , \cdots , d(s,s_n))$ for some $g: \KK_n^* \mapsto \RR$ vanishing at the point $\KK_n^*\backslash\KK_n$,  restricted to $\KK_n$. Then in the right hand side of (\ref{PT2.1Cc}), the first term is zero. Let $\mathcal{C}_n \subset C_b(\KK)$ denote the collections of such $f$, indexed by $n\geq 1$.

 Next, extend $H_f$ to $\KK^*$ by defining it to be $\liminf_{y'\in \KK,y'\to y}H_f(y')$ for $y\in\KK^*\backslash\KK$. By abuse of notation, we denote this extesnion by $H_f$ again. Note that $H_f$ is lower semicontinuous on $\KK^*$ by construction. Using Skorokhod's theorem, construct on some probability space $\KK^*$-valued random variables $\chi_n, n \geq 0,$ and  $\hat{\chi}$ such that the laws of $\chi_n$ (resp., $\hat{\chi}$) are $\widetilde\zeta_n$ (resp., $\widehat\zeta$) and $\chi_n \to \hat{\chi}$ a.s.


Then we have 
\begin{equation}\label{E02}
\begin{aligned}
\liminf_{n\to\infty}\,\int_{\KK}
\biggl(\int_{\XX} & f(y)\,P(\D{y}\smid x,u)\biggr)\,\widetilde\zeta_n(\D{x},\D{u})\\[5pt]
&\,=\, \liminf_{n\to\infty}\, \int_{\KK^*} 
H_f(x,u)\,\widetilde\zeta_n(\D{x},\D{u})\\[5pt]
&\,=\, \liminf_{n\to\infty}\,\Exp\bigl[H_f(\chi_n)\bigr]\\[5pt]
&\,\stackrel{(a)}{\ge}\,(1-a)\Exp\bigl[H_f(\hat{\chi})I_{\KK}\bigr] + a\Exp\bigl[H_f(\hat{\chi})I_{\KK^*\backslash\KK}\bigr] \\[5pt]
&\,\stackrel{(b)}{\ge}\,(1-a)\Exp\bigl[H_f(\hat{\chi})\bigr]\\[5pt]
&\,=\,(1-a)\int_{\KK} H_f(x,u)\,\zeta_1(\D{x},\D{u})\\[5pt]
&\,=\,
(1-a)\int_{\KK}
\biggl(\int_{\XX} f(y)\,P(\D{y}\smid x,u)\biggr)\,\zeta_1(\D{x},\D{u}),
\end{aligned}
\end{equation}
where `$(a)$' follows from the lower semicontinuity of $H_f$ and `$(b)$' follows from the fact that $H_f \geq 0$ on $\KK^*\backslash\KK$. Combining \cref{PT2.1B} with the above
and using  Fubini's theorem, we get
\begin{equation}\label{E04}
\int_{\KK} f(x)\,\zeta_1(\D{x},\D{u})\,\ge\,
\int_{\XX} f(y)\,\biggl(\int_{\KK} P(\D{y}\smid x,u)\,\zeta_1(\D{x},\D{u})\biggr)
\qquad\forall\,f\in \bm{F}\,.
\end{equation}
For $A\in\Borel(\XX)$, define
\begin{equation}\label{E-eta}
\begin{aligned}
\eta_1(A) &\,\df\,\int_{(A\times\Act)\cap\KK}\,\zeta_1(\D{x},\D{u})\,,\\[3pt]
\eta_2 (A) &\,\df\, \int_{\KK} P(A\smid x,u)\,\zeta_1(\D{x},\D{u})\,.
\end{aligned}
\end{equation}
Then, \cref{E04} can be written as
\begin{equation}\label{E06}
\int_{\XX} f(x)\,\eta_1(\D{x})\,\ge\,
\int_{\XX} f(y)\,\eta_2(\D{x})
\qquad\forall\,f\in \bm{F}\,.
\end{equation}
Such $f$ separate points of $\KK_n$ and therefore form a separating class for $\Pm(\KK_n)$. It follows that the set $\mathcal{C} := \cup_{n\geq 1}\mathcal{C}_n$ is a separating class for finite positive measures on $\KK$. Hence
\begin{equation}\label{E07}
\eta_1(B)\,\ge\, \eta_2(B)\,=\,
\int_{\KK} P(B\smid x,u)\,\zeta_1(\D{x},\D{u})
\qquad\forall\,B\in\Borel(\XX)\,.
\end{equation}
However, $\eta_1(\XX)=\eta_2(\XX)=1$ by \cref{E-eta}.
Thus equality must hold
in \cref{E07} for all $B\in\Borel(\XX)$, which means that
$\zeta_1\in\eom$ by the definition of $\eom$.

In the case of a sequence $\{\zeta_n\}_{n\in\NN}\subset\eom$ such that $\zeta_n\Rightarrow\widehat\zeta = a\zeta_0 + (1-a)\zeta_1$ as above,
observe that the left-hand side of \cref{PT2.1B} over
this sequence is identically equal to $0$ by the definition of
an ergodic occupation measure.
Thus, the proof of the statement is identical to the above.
\end{proof}

We continue by showing that \hyperlink{C}{\textup{(C)}} implies
the existence of an optimal ergodic occupation measure in the sense of \cref{D2.2}.

\begin{theorem}\label{T2.1}
Under \hyperlink{C}{\textup{(C)}}, we have $\eom^\star\ne\varnothing$.
In addition if $\zeta^\star\in\eom^\star$,
$\uppi_{\zeta^\star}\in\Pm(\XX)$ and $v_{\zeta^\star}$
satisfy \cref{E-disint}, and $\Hat{v}\in\Usm$ agrees
$\uppi_{\zeta^\star}$-a.e.\ 
with $v_{\zeta^\star}$, then
\begin{equation}\label{ET2.1A}
\lim_{N\to\infty}\,\frac{1}{N}\,
\Exp^{\Hat{v}}_x\Biggl[\sum_{n=0}^{N-1}\rc_{\Hat{v}}(X_n)\Biggr]\,=\,
J(x)\,
=\,\beta \quad\uppi_{\zeta^\star}\text{-a.e.\ }
\end{equation}
\end{theorem}

\begin{proof}
Let $\{\zeta_k\}_{k\in\NN}$ be such that
$\int\rc\,\D\zeta_k\searrow\beta$ as $k\to\infty$.
We select
a subsequence such that $\zeta_k\Rightarrow\widehat\zeta\in\Pm(\KK^*)$,
and write $\widehat\zeta=a\zeta'+(1-a)\zeta''$, with $a\in[0,1]$, $\zeta'\in\Pm(\partial\KK^*)$
and $\zeta''\in\Pm(\KK)$.

Since $c$ is lower semi-continuous on $\KK$, there exists a sequence $c_n \in C_b(\KK)$ such that $c_n \uparrow c$ pointwise. Choose a $n_0 \geq 1$ such that 
for $n \ge n_0$, $\{x_m\} \subset \KK, x_m \to \partial \KK := \KK^*\backslash\KK$, then
$$\liminf_{k\uparrow\infty}\inf_uc_n(x_k,u)  > \beta + 2\Tilde\varepsilon$$
for some  $\Tilde\varepsilon > 0$. Then we have
\begin{eqnarray}\label{PT2.1G}
\beta &\ge& \liminf_{k\to\infty}\,
\int_{\KK^*}\rc\, \D\zeta_k \,\ge\,\liminf_{k\to\infty}\,
\int_{\KK^*}\rc_n\, \D\zeta_k\, \nonumber \\
&\ge& a\int c_nd\zeta' + (1-a)\int c_nd\zeta" \,\ge\,
a (\beta + \Tilde\varepsilon) + (1-a) \int_{\KK}\rc_n\, \D\zeta''\,.
\end{eqnarray}
for all $n\ge n_0$.  By the above lemma, $\zeta'' \in \eom$, implying $\int_{\KK}\rc\, \D\zeta''\ge\beta$.
Letting $n\to\infty$ in (\ref{PT2.1G}), we obtain
\begin{equation*}
\begin{aligned}
\beta &\,\ge\,
a (\beta + \Tilde\varepsilon) + (1-a) \int_{\KK}\rc\, \D\zeta''\\[3pt]
&\,\ge\, a (\beta + \Tilde\varepsilon) +  (1-a) \beta\,.
\end{aligned}
\end{equation*}
This shows that $a=0$ and $\int_{\KK}\rc\, \D\zeta''=\beta$.
Therefore $\{\zeta_k\}$ are tight and  $\zeta''\in\eom^\star$, hence $\eom^\star \neq \phi$.

It remains to establish \cref{ET2.1A}.
If $\zeta^\star=\uppi_{\zeta^\star}\circledast v_{\zeta^\star}\in\eom^\star$
and $\Hat{v}\in\Usm$ agrees $\uppi_{\zeta^\star}$-a.e.\
with $v_{\zeta^\star}$, then an application of
Birkhoff's ergodic theorem  shows that
\begin{equation}\label{PT2.1H}
\beta\,=\,\int_\KK c\,\D\zeta^\star \,=\, \lim_{N\to\infty}\,\frac{1}{N}\,
\Exp^{\Hat{v}}_x\Biggl[\sum_{n=0}^{N-1}\rc_{\Hat{v}}(X_n)\Biggr]
\quad\uppi_{\zeta^\star}\text{-a.e.\ }
\end{equation}
This completes the proof.
\end{proof}

\begin{remark}
The pair $(\Hat{v},\uppi_{\zeta^\star})$ in \cref{T2.1} is a
stationary \emph{minimum pair}
in the sense of \cite[Definition~2.2]{Yu-20}
(see also \cite{Yu-19}).
It is worthwhile comparing the assumptions in \cite{Yu-20} to the ones in
this paper. 
In \cite{Yu-20} the state space $\XX$ is Borel, $\Act$ is countable,
and $\KK$ is a Borel subset of $\XX\times\Act$.
Existence of a stationary minimum pair is established under
the assumption that  $c$ is strictly unbounded
and the transition kernel satisfies a majorization condition.
The latter involves
weak continuity of the kernel $P$ and lower semi-continuity of
$c$ when these are restricted to $D\times\Act$, where $D\subset\XX$
is a closed set that appears in the majorization condition
\cite[Assumption~3.1]{Yu-20}.

By comparison, we allow $\Act$ to be Polish, the running
cost satisfies (C) and is not necessarily strictly
unbounded, and we don't need a majorization condition.
However, we assume
that $\XX$ is Polish, that $\cU$ is upper semicontinuous,
weak continuity of $P$ and lower semi-continuity of $c$ on
$\KK$,
which are more restrictive than \cite[Assumption~3.1]{Yu-20}.
\end{remark}

\subsection{Discussion}

To guide the reader in the approach we follow to establish
the Bellman equation and the existence of an optimal stationary Markov policy,
we review the case of an MDP on a countable state space with compact
action space under the near monotone hypothesis \cite{Borkar-91}.
Let the state space be $\NN_0\df\{0,1,2,\dotsc\}$,
and $\cU(x)=\Act$ for all $x\in\NN_0$. Suppose the state $0$ is reachable with positive probability from every other state under some control.
Under the near-monotone hypothesis in \hyperlink{C}{\textup{(C)}},
we obtain an optimal ergodic occupation measure
$\zeta^\star = \uppi_{\zeta^\star}\circledast v_{\zeta^\star}$.
Let $K\subset\NN_0$ denote the support of $\uppi_{\zeta^\star}$.
Then necessarily,  $0\in K$.
Then $v_{\zeta^\star}$ is defined on $K$ via the disintegration of $\zeta^\star$,
and thus the Markov chain `controlled' by $v_{\zeta^\star}$ is well defined when
restricted to $K$.
We would like to extend $v_{\zeta^\star}$ to some policy $v_\star\in\Usm$ which is
optimal in the sense of the definition in \cref{S2.1.1}.
Let $\uptau_A$ denote the first return time to a set $A$, defined by
\begin{equation*}
\uptau_A\,\df\, \min\,\{n\ge1\colon\, X_n\in A\}\,.
\end{equation*}
Let $\uptau_0 := \uptau_{\{0\}}$. A key observation is that $v_{\zeta^\star}$ satisfies
\begin{equation}\label{E-dscrA}
\Exp^{v_{\zeta^\star}}_x\Biggl[\sum_{n=0}^{\uptau_0-1}
\bigl(\rc_{v_{\zeta^\star}}(X_n)-\beta\bigr)\Biggr]
\,=\, \inf_{v\in\Usm}\,
\Exp^{v}_x\Biggl[\sum_{n=0}^{\uptau_0-1}\bigl(\rc_{v}(X_n)-\beta\bigr)\Biggr]
\qquad\forall\, x\in K\,.
\end{equation}
This can be shown by following the proof of \cref{L3.1} which establishes
an analogous result for the  model in this paper.
Therefore, any Markov control that arises from the disintegration of
an optimal ergodic occupation measure attains the infimum on the right-hand side of \cref{E-dscrA}
for $x\in K$.
Let
\begin{equation}\label{E-dscrB}
V(x)\,\df\,
\inf_{v\in\Usm}\,
\Exp^{v}_x\Biggl[\sum_{n=0}^{\uptau_0-1}\bigl(\rc_{v}(X_n)-\beta\bigr)\Biggr]\,,
\qquad x\in \NN_0\,.
\end{equation}
with $\beta$ as in \cref{ED2.2A}, and suppose that the
right-hand side of \cref{E-dscrB} is finite for all $x\in\NN$.
Then it is straightforward to show, using a one step analysis, that
$V$ satisfies
\begin{equation*}
V(x)\,=\,\min_{u\in\Act}\;\biggl[\rc(x,u)-\beta
+\sum_{y\in\NN_0\backslash\{0\}} V(y)\,P(y\,|\,x,u)\biggr]\qquad \forall\, x\in\NN\,,
\end{equation*}
in other words, we have the Bellman equation on the entire state space
except possibly at $x=0$.
Now, since $\beta$ is the ergodic value, we have
\begin{equation*}
\Exp^{v}_0\Biggl[\sum_{n=0}^{\uptau_0-1}\bigl(\rc_{v}(X_n)-\beta\bigr)\Biggr]\,\ge\,0\,,
\end{equation*}
with equality when $v=v_{\zeta^\star}$.
In particular, $V(0)=0$.
But this shows that the Bellman equation also holds for $x=0$.
One crucial step in this derivation is the finiteness of the right-hand side of
\cref{E-dscrB}.
Since $c-\beta\ge0$ on the complement of a finite set
by the near-monotone hypothesis, it is easy to show that
it suffices to assume that there exists some $v\in\Usm$ which satisfies
\begin{equation}\label{E-dscrE}
\Exp^{v}_x\Biggl[\sum_{n=0}^{\uptau_0-1}\rc_{v}(X_n)\Biggr]
\,<\,\infty\qquad\forall x\in\NN\,.
\end{equation}
The fact that $0$ is an atom  plays of course an important role in showing that
the Bellman equation is satisfied at $x=0$.
For the model in this paper, we circumvent this difficulty
by imposing a suitable hypothesis and adopting the
splitting method introduced by Athreya--Ney and Nummelin.
This is discussed in the next subsection.

\subsection{The split-chain and the pseudo-atom}\label{S2.4}

We introduce here the notions of the split chain and pseudo-atom,
originally due to Athreya and Ney \cite{Athreya-78},
and Nummelin \cite{Nummelin-78} for uncontrolled Markov chains.
We follow the treatment of \cite[Section~8.4]{ABG12}.
See \cite{Meyn-09} for an extended treatment, albeit in the uncontrolled framework.

The basic assumption concerns the existence of a $1$-small set
which is compatible with the near-monotonicity condition \hyperlink{C}{(C)}.
More precisely, the transition probability $P$ is assumed to satisfy the
following minorization hypothesis.
\begin{itemize}
\item[\hypertarget{A0}{\textbf{(A0)}}]
There exists a bounded set
 $\cB\subset\XX$ which satisfies
\begin{equation*}
\inf_{(x,u)\in(\cB^c\times\Act)\cap\KK}\;\rc(x,u)\,>\,\beta\,,
\end{equation*}
such that
for some measure $\nu\in\Pm(\XX)$ with $\nu(\cB)=1$, and a constant
$\delta > 0$, we have
$P(A\,|\,x,\cdot\,)\ge\delta\nu(A)\Ind_{\cB}(x)$ for all $A\in\Borel(\XX)$  and $\sum_{n\geq1}P(X_n\in \cB) > 0$ under all $v\in\Usm$.
Here, $\beta$ is as defined in \cref{ED2.2A}.
\end{itemize}

\begin{remark}
If $\XX=\Rd$ and the transition kernel has a continuous density $\varphi$,
then a necessary and sufficient condition
for the minorization condition in \hyperlink{A0}{(A0)} is that the function
$\varGamma\colon\cB\to\RR_+$
defined by
$$\varGamma(y) \,\df\, \inf_{(x,u)\in(\cB\times\Act)\cap\KK}\,\varphi(y\,|\,x,u)$$
is not equal to $0$ $\nu$-a.e.
In particular, if the density $\varphi$ is strictly positive,
then \hyperlink{A0}{(A0)} is automatically satisfied.
\end{remark}

\begin{definition}[Pseudo-atom]\label{D2.5}
Let 
\begin{equation*}
\cX \,\df\, (\XX\times\{0\})\cup(\cB\times\{1\})
\end{equation*}
and $\Borel(\cX)$ denote its Borel $\sigma$-algebra.
For a probability measure $\mu\in\Pm(\XX)$ we define
the corresponding probability measure $\Breve\mu$ on
$\Borel(\cX)$ by
\begin{equation}\label{ED2.5A}
\begin{aligned}
\Breve\mu(A\times\{0\}) &\,\df\, (1-\delta)\mu(A\cap\cB)
+\mu(A\cap\cB^c)\,,\quad A\in\Borel(\XX)\,,\\[3pt]
\Breve\mu(A\times\{1\}) &\,\df\,\delta\mu(A)\,,\quad A\in\Borel(\cB)\,.
\end{aligned}
\end{equation}
Let $\Breve\cB\df\cB\times\{1\}$, and refer to it as the \emph{pseudo-atom}.
\end{definition}

\begin{definition}[Split chain]\label{D2.6}
Given the controlled Markov chain
$\bigl(\XX,\Act,\cU,P,\rc\bigr)$ as described in \cref{S2},
we define the corresponding \emph{split chain}
$(\cX,\Act,\cU,Q,\Breve\rc)$,
with state space $\cX$, and transition kernel given by
\begin{equation}\label{E-Q}
Q(\D{y}\,|\, (x,i), u)\df\begin{cases}
\Breve{P}(\D{y}\,|\, x,u) &\text{if }(x,i)\in (\XX\times\{0\})\setminus(\cB\times\{0\})\,,
\\[5pt]
\frac{1}{1-\delta}\bigl(\Breve{P}(\D{y}\,|\, x,u) -\delta\Breve\nu(\D{y})\bigr) &
\text{if }(x,i)\in\cB\times\{0\}\,,\\[5pt]
\Breve\nu(\D{y}) &\text{if }(x,i)\in\cB\times\{1\}\,.
\end{cases}
\end{equation}
The running cost $\Breve\rc$ is defined in \cref{S4.4}.
\end{definition}

Using \cref{D2.5,E-Q}, the kernel $Q$ of the split chain can be expressed as follows:
\begin{equation}\label{E-Q1}
\begin{aligned}
Q(A\times\{0\}\,|\, (x,0), u)&\,\df\,
\bigl[P(A\cap\cB\,|\, x,u) -\delta\nu(A\cap\cB)
+\tfrac{1}{1-\delta}\,P(A\cap\cB^c\,|\, x,u)\bigr]\,\Ind_{\cB}(x)\\[4pt]
&\mspace{90mu}
+ \bigl[(1-\delta)P(A\cap\cB\,|\, x,u) + P(A\cap\cB^c\,|\, x,u)\bigr]
\,\Ind_{\cB^c}(x)
\end{aligned}
\end{equation}
for $A\in\Borel(\XX)$,
\begin{equation}\label{E-Q2}
Q(A\times\{1\}\,|\, (x,0), u)\,\df\,
\tfrac{\delta}{1-\delta}\bigl(P(A\,|\,x,u) -\delta\nu(A)\bigr)
\,\Ind_{\cB}(x)
+ \delta P(A\,|\, x,u)\,\Ind_{\cB^c}(x)
\end{equation}
for $A\in\Borel(\cB)$, and for $x \in \cB$,
\begin{equation}\label{E-Q3}
\begin{aligned}
Q(\D{y}\times\{0\}\,|\, (x,1), u)&\,\df\,
(1-\delta)\,\nu(\D{y})
\,,\\[5pt]
Q(\D{y}\times\{1\}\,|\, (x,1), u)&\,\df\,
\delta\,\nu(\D{y})\,.
\end{aligned}
\end{equation}
Note that $\cB^c\times\{1\}$ is not visited.

Given an initial distribution $\mu_0$ of $X_0$, the corresponding initial
distribution $\Breve\mu_0$ of the split chain is determined according to
\cref{ED2.5A}.
We let $\Breve X_n=(X_n, i_n)\in\cX$ denote the state process of the split chain.

Next we define an equivalent running cost for the split chain. 
Consider a function $\Breve\rc\colon\cX\times\Act\to\RR$ satisfying
\begin{equation*}
\begin{aligned}
\Breve\rc\bigl((x,0), u\bigr) &\,=\, \rc(x,u)\,,\qquad
(x,u)\in (\cB^c\times\Act)\cap\KK\,,\\[3pt]
\delta\Breve\rc\bigl((x,1), u\bigr)+
(1-\delta)\Breve\rc\bigl((x,0), u\bigr) &\,=\, \rc(x,u)\,,\qquad
(x,u)\in (\cB\times\Act)\cap\KK\,,
\end{aligned}
\end{equation*}
with $\Breve\rc\bigl((x,1), u\bigr)$ not depending on $u$.

Let $\Breve\Pm(\cX)$ denote the class of probability measures $\Breve\mu$
on $\Borel(\cX)$ which satisfy
$(1-\delta)\Breve\mu(A\times\{1\})=\delta\Breve\mu(A\times\{0\})$
for all $A\in\Borel(\cB)$.
It follows by \cref{ED2.6E}, that for any initial $\mu_0\in\Pm(\cX)$,
we have $\Breve\mu_0 Q_u\in\Breve\Pm(\cX)$.
In other words,  $\Breve\Pm(\cX)$ is invariant under the action of $Q$.
This property implies that
\begin{equation*}
\Breve\Exp^\xi_{\Breve\mu_0}\biggl[\sum_{n=0}^{N-1}
\Breve\rc(\Breve{X}_n, U_n)\biggr] \,=\,
\Exp^\xi_{\mu_0}\Biggl[\sum_{n=0}^{N-1}\rc(X_n, U_n)\Biggr]
\quad \forall\,\xi\in\Uadm\,.
\end{equation*}
In particular, the ergodic control problem of the split
chain under the cost-per-stage function $\Breve\rc$ is equivalent
to the original ergodic control problem.

With the above property in mind, we introduce the following definition.
\begin{definition}\label{D2.7}
We define the cost-per-stage function
$\Breve\rc\colon\cX\times\Act\to\RR$
for the split chain by
\begin{equation}\label{ED2.7A}
\begin{aligned}
\Breve\rc\bigl((x,0), u\bigr)&\,\df\,
\begin{cases}\rc(x,u) &\text{for all }x\in (\cB^c\times\Act)\cap\KK\,,\\[3pt]
\frac{\rc(x,u)}{1-\delta} &\text{for all }x\in (\cB\times\Act)\cap\KK\,,
\end{cases}\\[3pt]
\Breve\rc\bigl((x,1), u\bigr) &\,\df\, 0\qquad\forall\,x\in\cB\,.
\end{aligned}
\end{equation}
For $v\in\Usm$, we let $\Breve\rc_v$ be as in \cref{D2.1} with $\rc(\cdot)$ replaced by $\Breve\rc(\cdot)$.
\end{definition}

An equivalent description of the split chain is as follows. Let $\{\xi_n\}$ denote the control process. 
\begin{enumerate}

\item If $X_n = x \in \cB, \xi_n = u$ and $i_n = 0$, then $X_{n+1} = y$ according to the transition probability
$$\frac{1}{1-\delta}\bigl(P(\D{y}\,|\, x,u) -\delta\nu(\D{y})\bigr).$$
Furthermore, if $y \in \cB$, then $i_{n+1} = 1$ with probability $\delta$ and $= 0$ with probability $1-\delta$.

\item If  $X_n = x \in \cB$ and $i_n = 1$, then $X_{n+1} = y \in \cB$ with probability $\nu(dy)$ and $i_{n+1} = 1$ with probability $\delta$ and $= 0$ with probability $1-\delta$.

\item If $X_n = x \notin \cB$ and $i_n = 0$, then $X_{n+1} = y$ according to $P(dy|x,u)$ and $i_{n+1}$ is as in (1) above.

\item The set $\cB^c\times\{1\}$ is never visited.

\end{enumerate}

This gives a causal description of the split chain. We dub the control $\xi = \{\xi_n\}$ as an admissible control. Intuitively, it can depend at time $n$ on the past history till $n$, i.e., on $(X_m, i_m), m \leq n, \xi_k, k < n$, and in addition, on any extraneous randomization conditionally independent of the `future' $(X_m,i_m), \xi_m, m > n,$ given the history till $n$. 

It is clear that an admissible strategy $\xi\in\Uadm$,
or a  Markov randomized strategy $v=\{v_t\}_{t\in\NN_0}$
maps in a natural manner to a corresponding control for the split chain, which is also
denoted as $\xi$ or $v$, respectively.
We use the symbols,  $\Breve\Exp^\xi_{(x,i)}, \Breve\Exp^v_{(x,i)}$ to denote the expectation operator
on the path space of the split chain controlled under $\xi\in\Uadm, v \in \Uadm_{sm}$ resp., and adopt
the analogous notation as in \cref{D2.1}, e.g., $Q_v$ and
$\{\Breve X_n^v\}_{n\in\NN_0}$.
In addition, we let
\begin{equation}\label{E-uuptau}
\uuptau\,\df\,\min\{n \ge 1\colon\Breve{X}_n\in\Breve\cB\}\,,
\end{equation}
that is, the first return time to $\Breve\cB\df\cB\times\{1\}$.

Let
\begin{equation}\label{E-delta0}
\delta_\circ \,\df\,
\frac{1-\delta}{\delta}\,
\biggl(\inf_{(x,u)\in(\cB\times\Act)\cap\KK}\;P(\cB\,|\,x,u) - \delta\biggr)^{-1}\,.
\end{equation}
Since $\delta>0$ in \hyperlink{A0}{(A0)}
can always be chosen so that $(x,u)\mapsto P(\cB\,|\,x,u) - \delta$
is strictly positive on $(\cB\times\Act)\cap\KK$, we may assume that $\delta_\circ$ is a
(finite) positive constant.
We have the following simple lemma.

\begin{lemma}\label{L2.2}
For any $v\in\Usm$ it holds that
\begin{equation*}
\Breve\Exp^v_{(x,0)}\Biggl[\sum_{k=0}^{\uuptau-1} \Ind_{\cB\times\{0\}}
(\Breve{X}_k)\Biggr] \,\le\, \delta_\circ\,.
\end{equation*}
\end{lemma}

\begin{proof}
This follows directly from the fact that
$Q(\cB\times\{1\}\,|\, (x,0), u)\ge\delta_\circ^{-1}$ for $x\in\cB$ by
\cref{E-delta0,E-Q2}.
\end{proof}

Let $\mu_0$ be an initial distribution of $\{X_n\}_{n\in\NN_0}$.
Adopting the notation $Q_u(\,\cdot\,|\, z) = Q(\,\cdot\,|\, z,u)$ for $z\in\cX$,
an easy calculation using \cref{D2.6} shows that
$\Breve\mu_0 Q_u (\cdot) \df \int_\cX Q_u(\,\cdot\,|\, z)\,\Breve\mu_0(\D z)$ is given by
\begin{equation}\label{ED2.6E}
\begin{aligned}
\Breve\mu_0 Q_u (A\times\{0\})
&\,=\, \int_\XX\bigl[(1-\delta) P(A\cap\cB\,|\, x,u)
+ P(A\cap\cB^c\,|\, x,u)\bigr] \mu_0(\D{x})\,,\qquad A\in\Borel(\XX)\,,\\
\Breve\mu_0 Q_u (A\times\{1\})
&\,=\, \int_\XX \delta P(A\cap\cB\,|\, x,u) \mu_0(\D{x})\,,\qquad A\in\Borel(\cB)\,.
\end{aligned}
\end{equation}
It is important to note, as seen by \cref{ED2.6E},
 that the marginal of the law of $(\Breve{X}_n, U_n)$,
$n\ge 0$, on $(\KK)^{\infty}$ coincides with the law of $(X_n, U_n)$,
$n\ge 0$, but the split chain has a pseudo-atom $\cB\times\{1\}$ with many
desirable properties that will become apparent in the next section
(see \cite[Theorem~8.4.1, p.~289]{ABG12} and \cites{Athreya-78,Nummelin-78}).

\subsection{Some basic notions}

We now recall some standard background from the theory of
Markov chains on a general state space,
see, e.g., \cite{Meyn-09} for a more detailed treatment.
For $v\in\Usm$ we define the \emph{resolvent} $\cR_v$ by
\begin{equation*}
\cR_v(x,A) \,\df\,  \sum_{n=1}^\infty 2^{-n}\,P^n_v(x,A)\,.
\end{equation*}
Consider the chain $\{X_n\}_{n\ge0}$ controlled by $v\in\Usm$.
Recall that a measure $\psi$ on $\Borel(\XX)$ is called a
\emph{(maximal) irreducibility measure} for the chain if
$\psi$ is
absolutely continuous with respect to $\cR_v(x,\cdot)$ for all $x\in\XX$
(and $\psi$ is maximal among such measures).
In turn, the chain itself is said to be $\psi$-irreducible.
Let $\Borel^+(\XX)$ denote the class of Borel sets $A$ satisfying
$\psi(A)>0$.
Let $\uptau_A$ denote the first return time to a set $A$, defined by
\begin{equation*}
\uptau_A\,\df\, \min\,\{n\ge1\colon\, X_n\in A\}\,.
\end{equation*}
For a $\psi$-irreducible chain, a set $C$ is \emph{petite} if
there exists a positive constant $c$ such that
$\cR_v(x,A)\ge c\widehat{\psi}(A)$
every $A\in\Borel(\XX)$ and $x\in C$,  and some finite positive measure $\widehat{\psi}$ equivalent to $\psi$.
Recall also that a $\psi$-irreducible chain is called \emph{Harris}
if $\Prob_x(\uptau_A<\infty)=1$ for every $A\in\Borel^+(\XX)$ and $x\in\XX$,
and it is called \emph{positive Harris} if it admits an invariant probability
measure.

Let $f\colon \XX\to[1,\infty)$ be a measurable map.
For a $\psi$-irreducible chain, a set $D\in\Borel(\XX)$ is called
$f$-regular \cite{Meyn-09} if
\begin{equation*}
\sup_{x\in D}\;
\Exp_x\Biggl[\sum_{n=0}^{\uptau_A-1}f(X_n)\Biggr]\,<\,\infty\qquad
\forall\,A\in\Borel^+(\XX)\,.
\end{equation*}
If there is countable cover of $\XX$ by $f$-regular sets,
then the chain is called $f$-regular.
An $f$-regular chain is always positive Harris with a unique invariant
probability measure $\uppi$ and satisfies
\begin{equation*}
\lim_{N\to\infty}\;\frac{1}{N}\;\sum_{n=0}^{N-1}\Exp_x\bigl[
f(X_n)\bigr]\,=\, \uppi(f) \,\df\, \int_\XX f(x)\,\uppi(\D{x})\qquad\forall\,x\in\XX\,.
\end{equation*}

\section{The Bellman equation}\label{S3}

In view of the definitions of the preceding section, we lift the
control problem in \cref{S2.1.1} to an an equivalent problem
on the controlled split chain
$(\cX,\Act,\cU(x),Q,\Breve\rc)$ described in \cref{D2.6,D2.7}.
In other words, we seek to minimize over all admissible $\xi\in\Uadm$
the cost
\begin{equation*}
\limsup_{N\to\infty}\;\frac{1}{N}\;
\Breve\Exp_{(x,i)}^\xi\biggl[\sum_{n=0}^{N-1}
\Breve\rc(\Breve{X}_n, U_n)\biggr]\,.
\end{equation*}

\subsection{Two assumptions}\label{S3.1}

We need two additional assumptions.
To state the first,
we borrow the notion of $\cK$-inf-compactness from \cite{Feinberg-13}.
Recall that a function $f\colon S\to\RR$, where $S$ is a topological space
is called inf-compact (on $S$), if the set
$\{x\in S\colon f(x) \le \kappa\}$ (possibly empty) is compact in $S$
for all $\kappa\in\RR$.
A function $f\colon\KK\to\RR$ is called $\cK$-inf-compact if
for every compact set $K\subset\XX$ the function is inf-compact
on $(K\times\Act)\cap\KK$.

The first assumption
is a structural hypothesis on the running cost
and is stated as follows:
\begin{itemize}
\item[\hypertarget{A1}{\textbf{(A1)}}]
One of the following holds.
\begin{itemize}
\item[(i)]
For some $x\in\XX$, we have $J(x)<\infty$, and
the running cost $\rc$ is inf-compact on $\KK$.
\item[(ii)]
The running cost $\rc$ is $\cK$-inf-compact
and \hyperlink{C}{\textup{(C)}} holds.
\end{itemize}
\end{itemize}

It is clear that part~(i) of \hyperlink{A1}{(A1)}  implies \hyperlink{C}{(C)}.
Therefore, as shown in \cref{T2.1}, under \hyperlink{A1}{(A1)},
there exists an optimal ergodic occupation measure.

\begin{remark}
Hypothesis \hyperlink{A1}{(A1)}\,(i) cannot be satisfied unless
$\KK$ is $\sigma$-compact.
A non-trivial example of such a Polish space
is $\prod_{i\in\NN}
\{\lambda e_i \,\colon \lambda \geq 0\}$ where $\{e_i\}_{i\in\NN}$
is a complete orthonormal
basis for a Hilbert space with relative topology inherited from the ambient Hilbert space.
This space is not locally compact.
Note also that an inf-compact $\rc$ is automatically $\cK$-inf-compact
\cite[Lemma~2.1\,(ii)]{Feinberg-13}.
\end{remark}

The second assumption is analogous to \cref{E-dscrE} for denumerable MDPs.
We start with the following definition.

\begin{definition}
Let $\uuptau$ be as defined in \cref{E-uuptau}.
We say that $v\in\Usm$ is $\rc$-\emph{stable} if for the chain controlled by $v$
the map
\begin{equation*}
x\,\mapsto\,\Breve\Exp^{v}_{(x,0)}\Biggl[\sum_{k=0}^{\uuptau - 1}
\Breve\rc_{v}(\Breve{X}_k)\Biggr]
\end{equation*}
is locally bounded on $\XX$, and by that we mean that it is bounded
on every bounded set of $\XX$.
\end{definition}

We impose the following assumption.

\begin{itemize}
\item[\hypertarget{A2}{\textbf{(A2)}}]
There exists a $\rc$-\emph{stable} $v\in\Usm$.
\end{itemize}

 Assumptions \hyperlink{A0}{(A0)}--\hyperlink{A2}{(A2)} are in effect throughout
the rest of the paper, unless mentioned otherwise.
To see how they are used, consider the following.
Let $\Hat{v}\in\Usm$ be 
such that it
agrees $\uppi_{\zeta^\star}$-a.e.\ with the control $v_{\zeta^\star}$ obtained
via the disintegration of an optimal ergodic occupation measure
$\zeta^\star=\uppi_{\zeta^\star}\circledast v_{\zeta^\star}$,
whose existence is guaranteed by \hyperlink{A1}{(A1)}.
 It is then clear by \hyperlink{A0}{(A0)} and Proposition 5.1.1, p.\ 97, \cite{Meyn-09} that the chain controlled by $\Hat{v}$ is
$\nu$-irreducible and aperiodic.
Thus, the invariant probability measure $\uppi_{\zeta^\star}$ is unique
for the chain controlled by $\Hat{v}$ and is (trivially)  mutually absolutely continuous with respect to $\nu$ on its support.
This implies that $J(x,\Hat{v})=\beta$, a constant that does not depend on $x\in\XX$.
Compare this with the counterexample in 
\cite[Example~1, p.~178]{Dyn-Yush}.
Also, \hyperlink{A2}{(A2)} should be compared with part (b) of
\cite[Theorem~3.5]{Yu-20}.
It is clear that \hyperlink{A0}{(A0)}
implies that the split chain controlled by a $c$-stable $v\in\Usm$ is positive Harris.

In the rest of the paper
we let
\begin{equation}\label{E-zstar}
\zeta^\star\,=\,\uppi_{\zeta^\star}\circledast v_{\zeta^\star}\in\eom^\star
\end{equation}
be a generic optimal ergodic measure.
It is clear that
$\Breve\Exp^{v_{\zeta^\star}}_{(x,0)}\Bigl[\sum_{k=0}^{\uuptau - 1}
\Breve\rc_{v_{\zeta^\star}}({X}_k)\Bigr]$ is well defined
$\uppi_{\zeta^\star}$-a.e.

We continue with the following lemma.

\begin{lemma}\label{L3.1}
Any $c$-stable $v\in\Usm$ (and therefore every  $v\in\Usm$) satisfies
\begin{equation*}
\Breve\Exp_{(x,0)}^{v}\Biggl[\sum_{k=0}^{\uuptau - 1}
\bigl(\Breve\rc_v\bigl(\Breve{X}_k\bigr)-\beta\bigr)\Biggr]
\,\ge\, \Breve\Exp_{(x,0)}^{v_{\zeta^\star}}\Biggl[\sum_{k=0}^{\uuptau - 1}
\bigl(\Breve\rc_{v_{\zeta^\star}}\bigl(\Breve{X}_k\bigr)-\beta\bigr)\Biggr]
\qquad\uppi_{\zeta^\star}\text{-a.e.}
\end{equation*}
In particular, \hyperlink{A2}{(A2)} implies that
$x\mapsto \Breve\Exp^{v_{\zeta^\star}}_{(x,0)}\Bigl[\sum_{k=0}^{\uuptau - 1}
\Breve\rc_{v_{\zeta^\star}}({X}_k)\Bigr]$ is locally bounded
$\uppi_{\zeta^\star}$-a.e.
\end{lemma}

\begin{proof}
If not, then we have the reverse inequality  on some
set $A\in\Borel(\XX)$ with $\uppi_{\zeta^\star}(A)>0$, that is,
\begin{equation}\label{PL3.1A}
\Breve\Exp_{(x,0)}^{v}\Biggl[\sum_{k=0}^{\uuptau - 1}
\bigl(\Breve\rc_v\bigl(\Breve{X}_k\bigr)-\beta\bigr)\Biggr]
\,<\, \Breve\Exp_{(x,0)}^{v_{\zeta^\star}}\Biggl[\sum_{k=0}^{\uuptau - 1}
\bigl(\Breve\rc_{v_{\zeta^\star}}\bigl(\Breve{X}_k\bigr)-\beta\bigr)\Biggr]
\qquad\forall\,x\in A,
\end{equation}
with $+\infty$ a possible value for the right hand side. To simplify the expressions let
\begin{equation*}
\cJ(v) \,\df\, \sum_{k=0}^{\uuptau - 1}
\bigl(\Breve\rc_v\bigl(\Breve{X}_k\bigr)-\beta\bigr)\,,\qquad v\in\Usm\,.
\end{equation*}
 Since $\uppi_{\zeta^\star}(A)>0$, $\nu(A)> 0$ and $A$ is in the support of the resolvent $R_{\zeta^*}(x,\cdot)$ for $\pi_{\zeta^*}$-a.e. $x$.  Hence \cref{PL3.1A} implies that
\begin{equation}\label{PL3.1B}
\Breve\Exp_{(x,1)}^{v_{\zeta^\star}}\Bigl[\Ind_{\{\uuptau> \uptau_A\}}
\Breve\Exp_{\Breve{X}_{\uptau_A}}^{v_{\zeta^\star}}\bigl[\cJ(v_{\zeta^\star})\bigr]\Bigr]
\,>\,
\Breve\Exp_{(x,1)}^{v_{\zeta^\star}}\Bigl[\Ind_{\{\uuptau> \uptau_A\}}
\Breve\Exp_{\Breve{X}_{\uptau_A}}^{v}\bigl[\cJ(v)\bigr]\Bigr]\,.
\end{equation}
Consider
$\Tilde{v}=(\Tilde{v}_n\,,\;n\in\NN_0)$ defined by
Let 
\begin{equation*}
\Tilde{v}_n \,\df\,
\begin{cases}
v& \text{if\ \ } \uptau_A \le n < \uuptau\,,\\[2pt]
v_{\zeta^\star}&\text{otherwise.}
\end{cases}
\end{equation*}
It is clear that this can be extended to a (nonstationary) strategy
$\Tilde{v}\in\Uadm$ over the infinite horizon, by using  the $k^{\mathrm{th}}$ return time
to $\Breve\cB$, denoted as $\uuptau_k$, and the number of cycles $\varkappa(n)$ completed
at time $n\in\NN$, which is defined by
\begin{equation*}
\varkappa(n)\,\df\, \max\,\{k\colon n\ge \uuptau_k\}\,.
\end{equation*}
Using the strong Markov property and \cref{PL3.1B}, we obtain
\begin{align*}
0&\,=\, \Breve\Exp_{(x,1)}^{v_{\zeta^\star}}\bigl[\cJ(v_{\zeta^\star})\bigr]\\
&\,=\,
\Breve\Exp_{(x,1)}^{v_{\zeta^\star}}\Bigl[\Ind_{\{\uuptau\le \uptau_A\}}
\cJ(v_{\zeta^\star})\Bigr]
+ \Breve\Exp_{(x,1)}^{v_{\zeta^\star}}\Bigl[\Ind_{\{\uuptau> \uptau_A\}}
\Breve\Exp_{\Breve{X}_{\uptau_A}}^{v_{\zeta^\star}}\bigl[\cJ(v_{\zeta^\star})\bigr]\Bigr]\\
&\,>\,
\Breve\Exp_{(x,1)}^{\Tilde{v}}\Bigl[\Ind_{\{\uuptau\le \uptau_A\}}
\cJ(v_{\zeta^\star})\Bigr]
+ \Breve\Exp_{(x,1)}^{\Tilde{v}}\Bigl[\Ind_{\{\uuptau> \uptau_A\}}
\Breve\Exp_{\Breve{X}_{\uptau_A}}^{v}\bigl[\cJ(v)\bigr]\Bigr]\\
&\,=\,
\Breve\Exp_{(x,1)}^{\Tilde{v}}\Biggl[\sum_{k=0}^{\uuptau - 1}
\bigl(\Breve\rc_{\Tilde{v}}(\Breve{X}_k)-\beta\bigr)\Biggr]\,.
\end{align*}
For $m>0$,  let $\Breve\rc^m_{\Tilde{v}}\df \min\{m,\Breve\rc_{\Tilde{v}}\}$.
The preceding inequality shows that, for some $\varepsilon>0$, we have
\begin{equation}\label{PL3.1C}
\Breve\Exp_{(x,1)}^{\Tilde{v}}\Biggl[\sum_{k=0}^{\uuptau - 1}
\bigl(\Breve\rc^m_{\Tilde{v}}(\Breve{X}_k)-\beta\bigr)\Biggr]\,<\,
-\varepsilon\qquad\forall\, m\in\NN\,.
\end{equation}
We claim that \cref{PL3.1C} contradicts the fact that $\beta$ is the optimal ergodic
value.
Indeed, it is rather standard to show (see the proof
of Theorem~5.1 of \cite{Hasminskii}) that
\begin{equation*}
\frac{1}{T}\, \sum_{t=0}^{T-1}
\bigl(\Breve\rc^m_{\Tilde{v}}(\Breve{X}_t)-\beta\bigr)
\,\xrightarrow[T\to\infty]{}\,
\frac{1}{\Breve\Exp_{(x,1)}^{\Tilde{v}}[\uuptau]}\,
\Breve\Exp_{(x,1)}^{\Tilde{v}}\Biggl[\sum_{k=0}^{\uuptau - 1}
\bigl(\Breve\rc^m_{\Tilde{v}}(\Breve{X}_k)-\beta\bigr)\Biggr]
\qquad \Breve\Prob^{\Tilde{v}}\text{-a.s.},
\end{equation*}
which together with \cref{PL3.1C} implies (since $\Breve\rc^m_{\Tilde{v}}$ is bounded)
that
\begin{equation}\label{PL3.1D}
\lim_{T\to\infty}\,\frac{1}{T}\, \Breve\Exp_{(x,1)}^{\Tilde{v}}\Biggl[\sum_{t=0}^{T-1}
\bigl(\Breve\rc^m_{\Tilde{v}}(\Breve{X}_t)-\beta\bigr)\Biggr]
\,<\,-\varepsilon_1\qquad\forall\, m\in\NN\,,
\end{equation}
for some $\varepsilon_1>0$.
Let 
$c^m\df \min\{m,\rc\}$, and 
$\beta_m$ denote the optimal ergodic
value for $c^m$ in place of $\rc$, defined as in assumption \hyperlink{C}{(C)}.
We first show that $\beta_m\to\beta$ as $m\to\infty$.
By \cref{T2.1}, there exists $\zeta_m\in\eom$ such that
\begin{equation*}
\beta_m \,=\, \int_{\KK} \rc^m\,\D\zeta_m\qquad \forall\, m> \beta + 2\Tilde\varepsilon\,.
\end{equation*}
As argued in the proof of \cref{T2.1}, $\zeta_m$ converges along some subsequence
to a measure
$a\zeta'+(1-a)\zeta''\in\Pm(\KK^*)$, with $\zeta'\in\Pm(\partial\KK^*)$
and $\zeta''\in\Pm(\KK)$.
We employ a family $\{c^m_n\,,\, m,n\in\NN\}$
of lower semi-continuous functions on $\KK^*$ defined
as in \hyperlink{C}{(C)} with $\rc$ replaced by $\rc^m$.
Then, analogously to \cref{PT2.1G},
for any fixed $m>\beta + 2\Tilde\varepsilon$, we have
\begin{equation*}
\lim_{k\to\infty}\,\beta_k\,\ge\,
\lim_{k\to\infty}\,\int_{\KK} \rc^m_m\,\D\zeta_k
\,\ge\,
a(\beta+\Tilde\varepsilon)+(1-a) \int_{\KK} \rc^m_m\,\D\zeta''\,.
\end{equation*}
Taking limits as $m\to\infty$ and using monotone convergence,
we obtain
\begin{equation}\label{PL3.1E}
\beta\,\ge\,\lim_{k\to\infty}\,\beta_k\,\ge\,
a(\beta+\Tilde\varepsilon)+(1-a) \int_{\KK} \rc\,\D\zeta''\,.
\end{equation}
This implies that $a<1$, and therefore
$\zeta''\in\eom$ by \cref{L2.1}.
But then $\int_{\KK} c\,\D\zeta''\ge\beta$ by \cref{ED2.2A},
and the equality $\lim_{k\to\infty}\,\beta_k=\beta$ follows
from \cref{PL3.1E}.
Parenthetically, we mention that the above argument
also shows that the sequence $\{\zeta_m\}_{m\in\NN}$ is tight.
Continuing, \cref{PL3.1D} implies that
\begin{equation*}
\beta_m - \beta \,\le\,\lim_{T\to\infty}\,\frac{1}{T}\,
\Breve\Exp_{(x,1)}^{\Tilde{v}}\Biggl[\sum_{t=0}^{T-1}
\bigl(\Breve\rc^m_{\Tilde{v}}(\Breve{X}_t)-\beta\bigr)\Biggr]
\,=\,-\varepsilon_1\qquad\forall\, m\in\NN\,,
\end{equation*}
which is a contradiction. Note that for a stationary policy that is not $c$-stable, the claim is vacuously true because the left hand side of the inequality is $+\infty$.
This completes the proof.
\end{proof}

Let $v\in\Usm$ be $c$-stable.
It follows from \cref{L3.1} that the strategy
 which agrees with $v_{\zeta^\star}$ on the support
of $\uppi_{\zeta^\star}$ and with $v$ on its complement is also $c$-stable.


It is clear from the definition of $Q$ that the first exit distribution of
the split-chain from $\cB\times\{1\}$ does not depend on $x\in\cB$.
Thus $x\mapsto \Breve\Exp_{(x,1)}^{\Hat{v}}[\uuptau\mspace{2mu}]$ is constant on $\cB$.
This implies that, for all $f\in\Cc_b(\XX)$, with $\Breve{f}$ defined
analogously to \cref{ED2.7A} so that $\Breve{f}\bigl((x,1)\bigr)=0$ for all $x\in\cB$, we have
\begin{equation}\label{E-inv}
\uppi_{\zeta^\star}(f)\,\df\, \int_{\XX} f(y)\,\uppi_{\zeta^\star}(\D{y})
\,=\,\frac{\Breve\Exp_{(x,1)}^{\Hat{v}}\Bigl[\sum_{k=0}^{\uuptau - 1}
\Breve f(\Breve{X}_k)\Bigr]}
{\Breve\Exp_{(x,1)}^{\Hat{v}}[\uuptau\mspace{2mu}]}\qquad\forall\, x\in\cB\,.
\end{equation}
In fact, \cref{E-inv} holds for any $f\in \cL^1(\XX;\uppi_{\Hat{v}})$
by \cite[Proposition~5.9]{Nummelin-book}.
Therefore, we have
\begin{equation}\label{E3.2}
\Breve\Exp^{\Hat{v}}_{(x,1)}\Biggl[\sum_{k=0}^{\uuptau - 1}
\bigl(\Breve\rc_{\Hat{v}}(\breve{X}_k)-\beta\bigr)\Biggr]\,=\,0
\qquad\forall\, x\in\cB\,.
\end{equation}
It then follows by \cref{L3.1}
that the function
\begin{equation}\label{E3.3}
\Breve\cG_{\Hat{v}}^{(i)} (x)\,=\,\Breve\cG_{\Hat{v}} (x,i) \,\df\,
\Breve\Exp^{\Hat{v}}_{(x,i)}\Biggl[\sum_{k=0}^{\uuptau - 1}
\bigl(\Breve\rc_{\Hat{v}}(\breve{X}_k)-\beta\bigr)\Biggr]\,,\qquad(x,i)\in\cX\,,
\end{equation}
is locally bounded from above.
On the other hand, by \cref{L2.2} and the fact that
$\Breve\rc_{\Hat{v}}\ge\beta$ on $\cB^c\times\{0\}$
we have
$\inf_\XX\,\Breve\cG_{\Hat{v}}^{(0)}\ge - \beta \delta_\circ$.

In \cref{S3.2} we show that $\Breve\cG_{\Hat{v}}^{(i)} (x)$ solves
the Poisson equation.

\subsection{Solution to the Poisson equation}\label{S3.2}

Let ${\Hat{v}}\in\Usm$ be 
$\rc$-\emph{stable}, and 
such that it
agrees $\uppi_{\zeta^\star}$-a.e.\ with the control $v_{\zeta^\star}$ obtained
via the disintegration of an optimal ergodic occupation measure
$\zeta^\star=\uppi_{\zeta^\star}\circledast v_{\zeta^\star}$.
By one step analysis, using \cref{E-Q1,E-Q2,E-Q3,ED2.7A,E3.3}, 
adopting the notation in \cref{D2.1}, we obtain
\begin{equation}\label{E-stp1a}
\Breve\cG_{\Hat{v}}^{(1)}(x) \,=\, -\beta +
 (1 -\delta)\int_\cB\Breve\cG_{\Hat{v}}^{(0)}(y)\nu(\D{y})
 + \delta\int_\cB\Breve\cG_{\Hat{v}}^{(1)}(y)\nu(\D{y})\,,\quad x\in\cB\,,
\end{equation}
\begin{equation}\label{E-stp1b}
\begin{aligned}
\Breve\cG_{\Hat{v}}^{(0)}(x) &\,=\, \tfrac{\rc_{\Hat{v}}(x)}{1-\delta}  -\beta
+ \int_\cB\Breve\cG_{\Hat{v}}^{(0)}(y)\bigl[P_{\Hat{v}}(\D{y}\,|\,x)-\delta\nu(\D{y})\bigr]
+\tfrac{1}{1-\delta}\int_{\cB^c}\Breve\cG_{\Hat{v}}^{(0)}(y)P_{\Hat{v}}(\D{y}\,|\,x)\\
&\mspace{200mu} +
\tfrac{\delta}{1-\delta}\int_{\cB}\Breve\cG_{\Hat{v}}^{(1)}(y)
\bigl[P_{\Hat{v}}(\D{y}\,|\,x)-\delta\nu(\D{y})\bigr]\,,
\quad x\in\cB\,,
\end{aligned}
\end{equation}
and
\begin{equation}\label{E-stp1c}
\begin{aligned}
\Breve\cG_{\Hat{v}}^{(0)}(x) &\,=\, \rc_{\Hat{v}}(x)  -\beta
+ (1-\delta)\int_\cB\Breve\cG_{\Hat{v}}^{(0)}(y)P_{\Hat{v}}(\D{y}\,|\,x)
+\int_{\cB^c}\Breve\cG_{\Hat{v}}^{(0)}(y)P_{\Hat{v}}(\D{y}\,|\,x)\\
&\mspace{200mu} + \delta \int_\cB\Breve\cG_{\Hat{v}}^{(1)}(y)P_{\Hat{v}}(\D{y}\,|\,x)
\,,\quad x\in\cB^c\,.
\end{aligned}
\end{equation}
Let
\begin{equation}\label{E-overc}
\overline\rc(x,u) \,\df\, \rc(x,u) - \beta\,,\qquad\text{and}\quad
\overline\rc_{\Hat{v}}(x) \,\df\, \rc_{\Hat{v}}(x) - \beta\,.
\end{equation}
Multiplying \cref{E-stp1a} and \cref{E-stp1b} by $\delta$ and $(1-\delta)$,
respectively, and adding them together, we obtain
\begin{equation}\label{E-stp2}
\begin{aligned}
(1 -\delta)\Breve\cG_{\Hat{v}}^{(0)}(x)+\delta\,\Breve\cG_{\Hat{v}}^{(1)}(x)
&\,=\, \overline\rc_{\Hat{v}}(x)
+ \int_{\cB}\bigl[(1 -\delta)\Breve\cG_{\Hat{v}}^{(0)}(y)
+\delta\,\Breve\cG_{\Hat{v}}^{(1)}(y)\bigr]P_{\Hat{v}}(\D{y}\,|\,x)\\
&\mspace{150mu} + \int_{\cB^c}\Breve\cG_{\Hat{v}}^{(0)}(y)P_{\Hat{v}}(\D{y}\,|\,x)\,,
\quad x\in\cB\,.
\end{aligned}
\end{equation}

We define
\begin{equation}\label{E-cGB}
\cG_{\Hat{v}}(x)\,\df\,\begin{cases}
(1 -\delta)\Breve\cG_{\Hat{v}}^{(0)}(x)+\delta\,\Breve\cG_{\Hat{v}}^{(1)}(x)\,,
&\text{for\ }x\in\cB\,,\\[5pt]
\Breve\cG_{\Hat{v}}^{(0)}(x), &\text{otherwise}.
\end{cases}
\end{equation}

It follows by \cref{E-stp1c,E-stp2,E-cGB} that
\begin{equation}\label{E-stp3}
\cG_{\Hat{v}}(x)\,=\,\overline\rc_{\Hat{v}}(x) +\int \cG_{\Hat{v}}(y)\,P_{\Hat{v}}(\D{y}\,|\,x)
\,=\, \overline\rc_{\Hat{v}}(x) + P_{\Hat{v}} \cG_{\Hat{v}}(x)\,,\qquad x\in\XX\,.
\end{equation}

It is clear that \cref{E3.2} implies that
$\Breve\cG_{\Hat{v}}^{(1)}\equiv 0$ on $\cB$.
Thus
\begin{equation*}
\int_\cB\cG_{\Hat{v}}(y)\nu(\D{y}) \,=\,
\int_\cB (1-\delta)\Breve\cG_{\Hat{v}}^{(0)}(y)\nu(\D{y})\,=\,\beta
\end{equation*}
by \cref{E-stp1a,E-cGB}.

Note that
$\Breve\cG_{v_{\zeta^\star}}^{(i)}$ and
$\cG_{v_{\zeta^\star}}$ are well defined $\uppi_{\zeta^\star}$-a.e.\
via \cref{E3.3,E-cGB}.

\subsection{Derivation of the Bellman equation (ACOE)}\label{S3.3}

Starting in this section, and throughout the rest of the paper,
we enforce the following structural hypothesis on the controlled chain.
This assumption is implicit in all the results of the paper which follow,
unless otherwise mentioned.

\begin{assumption}\label{A3.1}
$P(\D{y}\,|\, x,u)$ is \emph{strongly continuous} (or strong Feller),
that is, the map $\KK\ni(x,u)\mapsto\int_\XX f(y)P(\D{y}\,|\, x,u)$
is continuous for every $f\in\cM_b(\XX)$.
\end{assumption}

\begin{remark}
\cref{A3.1} implies that
the family
$\{P(\,\cdot\,|\, x,u)\,\colon (x,u)\in K\}$ is tight
for any compact set $K\subset\KK$.
Indeed, for any sequence $(x_n,u_n)\in K$ converging to
some $(x,u)$ in this set, we have $P(\,\,\cdot\,|\, x_n,u_n)\Rightarrow P(\,\cdot\,|\, x,u)$.
Then the above set, being the continuous image of a compact set, is compact.
By Prokhorov's theorem, it is tight.
\end{remark}

\begin{remark}
If $\XX=\Rd$, a sufficient condition for \cref{A3.1} is that
$$P(\D{y}\,|\, x,u) \,=\, \varphi(y\,|\, x,u)\lambda(\D{y})$$
for a density function $\varphi$
with respect to $\lambda$, the Lebesgue measure on $\Rd$,
and that the map
\begin{equation*}
(y,x,u)\in\Rd\times\KK\mapsto\varphi(y\,|\,x,u)\in [0,\infty)
\end{equation*}
is continuous.
This implies the continuity of the measure-valued map $(x,u) \mapsto \varphi(y|x,u)dy$ in total variation norm by Scheffe's theorem, which in turn implies \cref{A3.1}.
\end{remark}

\begin{definition}\label{D3.2}
Define 
\begin{equation*}
\Breve{V}_\star^{(0)}(x) \,\df\, \inf_{v\in\Usm}\,
\Breve\Exp^{v}_{(x,0)}\Biggl[\sum_{k=0}^{\uuptau - 1}
\bigl(\Breve\rc_{v}(\Breve{X}_k)-\beta\bigr)\Biggr]\,,\qquad x\in\XX\,,
\end{equation*}
Also let $\Breve{V}_\star^{(1)}(x)=0$ for $x\in\cB$,
and
\begin{equation*}
V_\star(x)\,\df\,\begin{cases}
(1 -\delta)\Breve{V}_\star^{(0)}(x)\,,
&\text{for\ }x\in\cB\,,\\[5pt]
\Breve{V}_\star^{(0)}(x), &\text{otherwise}.
\end{cases}
\end{equation*}
\end{definition}

Recall the definition of
$\order(f)$ in \cref{S-not},
and that $\fL(\XX)$ denotes the class of real-valued lower semi-continuous functions
which are bounded from below in $\XX$.

\begin{theorem}\label{T3.1}
The function $V_\star$ in \cref{D3.2} is in the class
$\fL(\XX)$ and satisfies
\begin{equation}\label{ET3.1A}
V_\star(x)\,=\,\min_{u\in\cU(x)}\;\biggl[\overline\rc(x,u)
+\int_\XX V_\star(y)\,P(\D{y}\,|\,x,u)\biggr]\qquad \forall\, x\in\XX\,,
\end{equation}
with $\overline\rc$ as in \cref{E-overc}.
Moreover, every $v_\star\in\Usm$ which satisfies
\begin{equation}\label{E-veri}
v_\star(x)\;\in\;\Argmin_{u\in\cU(x)}\;\bigl[\overline\rc(x,u)
+P_u V_\star(x)\bigr]
\end{equation}
is an optimal stationary Markov strategy.
In addition, \cref{ET3.1A} has, up to an additive constant, a unique solution
in $\fL(\XX)\cap\order(V_\star)$.
\end{theorem}

\begin{proof}
As in \cref{E-stp1b,E-stp1b},  using standard dynamic programming arguments in place of  one step analysis, we obtain
\begin{equation}\label{PT3.1A}
\begin{aligned}
\Breve{V}_\star^{(0)}(x) &\,=\, \min_{u\in\cU(x)}\,
\biggl[\tfrac{\rc(x,u)}{1-\delta}  -\beta
+ \int_\cB\Breve{V}_\star^{(0)}(y)\bigl[P(\D{y}\,|\,x,u)-\delta\nu(\D{y})\bigr]\\
&\mspace{250mu}
+\tfrac{1}{1-\delta}\int_{\cB^c}\Breve{V}_\star^{(0)}(y)P(\D{y}\,|\,x,u)\biggr]\,,
\quad x\in\cB\,,
\end{aligned}
\end{equation}
and
\begin{equation}\label{PT3.1B}
\begin{aligned}
\Breve{V}_\star^{(0)}(x) &\,=\, \min_{u\in\cU(x)}\,
\biggl[\rc(x,u)  -\beta
+ (1-\delta)\int_\cB\Breve{V}_\star^{(0)}(y)P(\D{y}\,|\,x,u)
+\int_{\cB^c}\Breve{V}_\star^{(0)}(y)P(\D{y}\,|\,x,u)\biggr]
\end{aligned}
\end{equation}
for $x\in\cB^c$.  

On the other hand, since $\Breve{V}_\star^{(0)}=\Breve\cG_{v_{\zeta^\star}}^{(0)}$
$\nu$-a.e. by \cref{L3.1}, then \cref{E-stp1a} shows that
\begin{equation}\label{PT3.1C}
0\,=\,\Breve{V}_\star^{(1)}(x) \,=\, -\beta +
 (1 -\delta)\int_\cB\Breve{V}_\star^{(0)}(y)\nu(\D{y})\quad \forall\,x\in\cB\,.
\end{equation}
It then follows by \cref{PT3.1A,PT3.1B,PT3.1C} and
\cref{D3.2}, that
$V_\star$ satisfies
\begin{equation}\label{PT3.1D}
V_\star(x)\,=\,\min_{u\in\cU(x)}\;\bigl[\overline\rc(x,u)
+P_u V_\star(x)\bigr]\,.
\end{equation}
Since the kernel $P$ is strongly continuous
and $V_\star$ is bounded from below in $\XX$ by \cref{L2.2},
the map $(x,u)\mapsto P_u V_\star(x)$ is 
lower semi-continuous on $\KK$.
Therefore, since $\cU$ is upper semi-continuous, the map
$(x,u)\mapsto \overline\rc(x,u)+P_u V_\star(x)$ is $\cK$-inf-compact by
\cite[Lemma~2.1\,(i)]{Feinberg-13}.
Hence, applying Theorem~2.1 of \cite{Feinberg-13} to \cref{PT3.1D},
we deduce that  $V_*\in\fL(\XX)$.
Since $V_\star$ is bounded from below in $\XX$ by \cref{L2.2}, existence and optimality
of $v_\star$ in \cref{E-veri} follows by a standard argument using Birkhoff's ergodic
theorem.

We continue with the proof of uniqueness.
Since $\Breve{V}_\star^{(0)}$ is bounded from below in $\XX$, it is standard to show,
using \cref{PT3.1B} and Fatou's lemma, that
\begin{equation}\label{PT3.1E}
\Breve{V}_\star^{(0)}(x) \,\ge\,
\Breve\Exp^{v_\star}_{(x,0)}\Biggl[\sum_{k=0}^{\uuptau - 1}
\bigl(\Breve\rc_{v_\star}(\Breve{X}_k)-\beta\bigr)\Biggr]\,,\qquad x\in\XX\,,
\end{equation}
with $v_\star$ as in \cref{E-veri}.
\cref{D3.2} shows that we must have equality in \cref{PT3.1E}.
In turn, applying Dynkin's formula to \cref{PT3.1B} we obtain
\begin{equation*}
\Breve{V}_\star^{(0)}(x) \,=\, \lim_{n\to\infty}\,
\Breve\Exp^{v_\star}_{(x,0)}\Biggl[\sum_{k=0}^{\uuptau\wedge n - 1}
\bigl(\Breve\rc_{v_\star}(\Breve{X}_k)-\beta\bigr)\Biggr]\,,\qquad x\in\XX\,,
\end{equation*}
and
\begin{equation}\label{PT3.1F}
\limsup_{n\to\infty}\, \Breve\Exp^{v_\star}_{(x,0)}\Bigl[
\Breve{V}_\star^{(0)} (\Breve{X}_{\uuptau})\Ind_{\{\uuptau> n\}}\Bigr]\,=\,0\,.
\end{equation}

Let $V\in\fL(\XX)\cap\order(V_\star)$ be a solution of \cref{ET3.1A}
and $\Hat{v}\in\Usm$ a selector from its minimizer.
Going to the split chain and scaling with an additive constant, we obtain
functions $\Breve{V}^{(i)}(x)$, $i=0,1$, which satisfy \cref{PT3.1A,PT3.1B,PT3.1C}
(with $\Breve{V}_\star^{(i)}$ replaced by $\Breve{V}^{(i)}$), and
\begin{equation*}
V(x)\,\df\,\begin{cases}
(1 -\delta)\Breve{V}^{(0)}(x)\,,
&\text{for\ }x\in\cB\,,\\[5pt]
\Breve{V}^{(0)}(x), &\text{otherwise}.
\end{cases}
\end{equation*}
In analogy to \cref{PT3.1E}, we also have
\begin{equation}\label{PT3.1G}
\Breve{V}^{(0)}(x) \,\ge\,
\Breve\Exp^{\Hat{v}}_{(x,0)}\Biggl[\sum_{k=0}^{\uuptau - 1}
\bigl(\Breve\rc_{\Hat{v}}(\Breve{X}_k)-\beta\bigr)\Biggr]
\,\ge\, \Breve{V}_\star^{(0)}(x)\,,\qquad x\in\XX\,,
\end{equation}
where for the second inequality we use \cref{D3.2}.
Thus, if $v_\star$ is as in \cref{E-veri}, then using the kernel $Q_{v_\star}$ of
the split chain in \cref{E-Q}, we deduce that
$\Breve{V}^{(i)}-\Breve{V}_\star^{(i)}$ is a nonnegative local supermartingale under
$Q_{v_\star}$.
Since $\Breve{V}^{(1)}=\Breve{V}_\star^{(1)}=0$,
and $V\in\order(V_\star)$, using Dynkin's formula,
we obtain from \cref{PT3.1F} and the supermartingale
inequality that
\begin{equation}\label{PT3.1H}
\Breve{V}^{(0)}(x)-\Breve{V}_\star^{(0)}(x)\,\le\,0\qquad\forall\,x\in\XX\,.
\end{equation}
Therefore,
$\Breve{V}^{(0)}=\Breve{V}_\star^{(0)}$ on $\cX$ by \cref{PT3.1G,PT3.1H}.
This completes the proof.
\end{proof}

\begin{remark}
If we relax the strong Feller hypothesis in \cref{A3.1}, and assume instead that
the transition kernel is weak Feller, we can obtain an ACOI
with a lower semi-continuous potential function.
Indeed, if we let
\begin{equation*}
\widetilde{V}_\star(x) \,\df\, \sup_{r>0}\, \inf_{y\in B_r(x)}\, V_\star(y)\,,
\end{equation*}
with $B_r(x)$ denoting the open ball of radius $r$ centered at $x$,
then $\widetilde{V}_\star\in\fL(\XX)$.
Therefore, by \cref{E-stp3} we have
\begin{equation}\label{ER3.5A}
V_\star(x)\,\ge\, \inf_{u\in\cU(x)}\,\bigl[ \overline\rc(x,u)
+ P_u V_\star(x)\bigr]\,\ge\,
\inf_{u\in\cU(x)}\,\bigl[ \overline\rc(x,u)+ P_u \widetilde{V}_\star(x)\bigr]\,,
\end{equation}
and the term on the right-hand side of \cref{ER3.5A} is in $\fL(\XX)$.
Since $\widetilde{V}_\star$ is the largest lower semi-continuous function
dominated by $V_\star$ \cite{Vega-Amaya-18}, we obtain
\begin{equation*}
\widetilde{V}_\star(x)\,\ge\,
\inf_{u\in\cU(x)}\,\bigl[ \overline\rc(x,u)+ P_u \widetilde{V}_\star(x)\bigr]\,.
\end{equation*}
It is standard to show that any measurable selector from this equation is optimal.
We refer the reader to \cites{Jaskiewicz-06,Vega-Amaya-18} on how to improve
this to an ACOE under additional hypotheses.
\end{remark}

\begin{remark}
Our approach differs from the standard approach of deriving the Bellman equation
using a vanishing discount argument. We briefly indicate here how near-monotonicity
or inf-compactness of the cost function can help us with the standard methodology.
One important consequence of near-monotonicity is that we can prove that the
discounted value function attains its minimum on a fixed compact set as the
discount parameter varies.
Thus if we can establish equicontinuity of the relative discounted value functions
as the discount factor varies (e.g., using convexity when available as in \cite{FAM-92},
or in \cref{Ex6.1} in \cref{S6}),
one can argue that as the discount parameter tends to $1$, the relative discounted
value functions either remain bounded on compacts or tend to infinity uniformly
on compacts along a subsequence.
Eliminating the latter possibility by a suitable choice of the offset in the
definition of the relative discounted value function shows uniform boundedness
over compacts.
This idea is used in \cite{Agarwal-08} for deriving the Bellman equation for
the average cost for a specific class of problems, and can potentially be generalized.
See also \cite[Theorem~6]{Feinberg-12}.
\end{remark}

\begin{remark}
It is  worth noting that the derivation of the Bellman equation \cref{PT3.1D}
does not require the strong continuity of \cref{A3.1}; weak continuity will suffice.
We do, however, require strong continuity in order to obtain
a solution in the class $\fL(\XX)$.
\end{remark}

\section{The value iteration}\label{S4}

Throughout this section as well as \cref{S6}, $v_\star\in\Usms$ is
some optimal stationary Markov strategy which is kept fixed.

\subsection{The value iteration algorithm}

We start with the following definition.

\begin{definition}
\textbf{[Value Iteration]}
Given $\Phi_0\in\fL(\XX)$ which serves as an initial condition, we define
the value iteration (VI) by
\begin{equation}\label{E-VI}
\Phi_{n+1}(x)\,=\,\overline\cT\Phi_{n}(x)\,\df\,\min_{u\in\cU(x)}\;\bigl[\overline\rc(x,u)
+P_u\Phi_{n}(x)\bigr],\quad n\in\NN_0\,.
\end{equation}
\end{definition}

Since $\overline\cT\colon \fL(\XX)\to\fL(\XX)$, it is clear that
the algorithm lives in the space of lower semi-continuous functions
which are bounded from below in $\XX$.
It is also clear that $\overline\cT$ is a monotone operator on $\fL(\XX)$,
that is, for any $f, f'\in\fL(\XX)$ with $f\le f'$, we have
$\overline\cT f \le \overline\cT f'$.

\subsubsection{The value iteration for the split chain}
Using  \cref{E-Q1,E-Q2,E-Q3} we can also express the algorithm via
the split chain as follows.
The \emph{value iteration functions}
$\bigl\{\Breve\Phi_n^{(i)}\,,\,n\in\NN_0\,,i=0,1\bigr\}$, are defined as follows.
Let $V_0\colon\XX\to\RR$ be a nonnegative continuous function.
The initial condition is $\Breve\Phi_0^{(1)}=0$,
and $\Breve\Phi_0^{(0)}(x) = V_0(x)\bigl[ (1-\delta)^{-1}\Ind_{\cB}(x)+
\Ind_{\cB^c}(x)\bigr]$,
and for each $n\in\NN$, define
\begin{equation}\label{E-Phispl}
\Phi_n(x)\,=\,
\begin{cases}
(1 -\delta)\Breve\Phi_n^{(0)}(x) +\delta\Breve\Phi_n^{(1)}\,,
&\text{for\ }x\in\cB\,,\\[5pt]
\Breve\Phi_n^{(0)}(x), &\text{otherwise}.
\end{cases}
\end{equation}
Thus, the algorithm takes the form
\begin{eqnarray}
\Breve\Phi_{n+1}^{(0)}(x) &=& 
\frac{1}{1 -\delta}\,\min_{u\in\cU(x)}\;\biggl[\overline\rc(x,u)
+ \int_{\XX}\Phi_{n}(y)\,P(\D{y}\,|\,x,u)\biggr] \nonumber \\
&& - \ \frac{\delta}{1 -\delta}\int_\cB \Phi_n(y)\,\nu(\D{y})\,,\quad x\in\cB\,, \label{one}
\\
\Breve\Phi_{n+1}^{(0)}(x) &=& \min_{u\in\cU(x)}\;\biggl[\overline\rc(x,u)
+ \int_{\XX}\Phi_{n}(y)\,P(\D{y}\,|\,x,u)\biggr]\,,\quad x\in\cB^c\,, \label{two}
\\
\Breve\Phi_{n+1}^{(1)}(x) &=& -\beta+
 \int_\cB \Phi_n(y)\,\nu(\D{y})\,,\qquad x\in\cB\,. \label{three}
\end{eqnarray}

\medskip

\begin{notation}\label{N4.1}
We adopt the following simplified notation.
We let $\Hat{v}_n\in\Usm$ be
a measurable selector from the minimizer of  \cref{E-VI}, and
define
\begin{equation}\label{EN4.1A}
\widehat{P}_n(\cdot\,|\,x) \,\df\, P\bigl(\cdot\,|\,x,\Hat{v}_n(x)\bigr)\,,\quad\text{and}\quad
\widehat{\rc}_n(x) \,\df\, \rc(x, \Hat{v}_n(x))-\beta\,.
\end{equation}
Note that these depend on the initial value $\Phi_0$.

We fix an optimal strategy $v_\star\in\Usms$, and
let $P_{\star}(\D{y}\,|\,x)$ denote the transition kernel
under $v_\star$.
In addition, we let $\rc_\star(x) = \rc\bigl(x,v_\star(x)\bigr)$
and $\overline\rc_\star = \rc_\star-\beta$.
\end{notation}
With this notation, for $n\in\NN_0$, we have
\begin{equation}\label{ELC01}
\begin{aligned}
\Phi_{n+1}(x) &\,=\,\min_{u\in\cU(x)}\,
\bigl[\overline\rc(x,u) +P_u\Phi_{n}(x)\bigr]\\[5pt]
&\,=\,\widehat{\rc}_n(x) +\widehat{P}_n \Phi_{n}(x)\qquad\forall x\in\XX\,,
\end{aligned}
\end{equation}
and
\begin{equation}\label{ELC02}
V_{\star}(x)\,=\,\overline\rc_\star(x) +P_\star V_{\star}(x)\qquad\forall x\in\XX\,.
\end{equation}
It follows from optimality of $v_{\star}$ and $\Hat{v}_n$ that
\begin{equation}\label{ELC03}
\Phi_{n+1}\,\le\,\overline\rc_\star +P_\star \Phi_n\,,
\end{equation}
and
\begin{equation}\label{ELC04}
V_{\star}\,\le\,\widehat{\rc}_n + \widehat{P}_n V_{\star}\,.
\end{equation}

\subsection{General results on convergence of the VI}\label{S4.2}

Recall the function $V_\star$ from \cref{T3.1}, and let
$\uppi_*$ denote the associated invariant probability measure.
Consider the following hypothesis.

\medskip

\begin{itemize}
\item[\hypertarget{H1}{\textbf{(H1)}}]
$\uppi_\star(V_\star)<\infty$.
\end{itemize}

\medskip

For $\rc$ bounded, finiteness of the second moments of $\uuptau$ implies
\hyperlink{H1}{(H1)} (see, for example, \cite[p.~66]{Borkar-91}).
In general, \hyperlink{H1}{(H1)} is equivalent to the finiteness of
the second moments
of the modulated first hitting times to $\cB\times\{1\}$ on a full and absorbing
set.

For a constant $\kappa\in\RR$ we define the set
\begin{equation}\label{E-cV}
\cV(\kappa) \,\df\, \bigl\{f\in\fL(\XX)\cap\order(V_\star)\,\colon
f\ge V_\star-\kappa\,,\ \uppi_\star(f) \le \kappa+1\bigr\}\,.
\end{equation}

Under \hyperlink{H1}{(H1)},
 we show that the VI converges pointwise for any $\Phi_0\in\cV(\kappa)$.
In order to prove this result, we need the following lemma.

\begin{lemma}\label{L4.1}
Under \textup{\hyperlink{H1}{(H1)}},
if $\Phi_0\in\cV(\kappa)$ for some $\kappa\in\RR$, then $\Phi_n\in\cV(\kappa)$
for all $n\in\NN$, or in other words, the set
$\cV(\kappa)$ is invariant under the action of $\overline\cT$.
In addition, $\uppi_\star(\Phi_{n+1}) \le \uppi_\star(\Phi_{n})$ for all
$n\in\NN_0$.
\end{lemma}

\begin{proof}
Subtracting \cref{ELC02} from \cref{ELC03} we obtain
\begin{equation}\label{EL4.1A}
\Phi_{n+1} - V_{\star}\,\le\, P_{\star}(\Phi_n - V_{\star})\,,
\end{equation}
while by subtracting \cref{ELC04} from \cref{ELC01} we have
\begin{equation}\label{EL4.1B}
\Phi_{n+1}-V_{\star} \,\ge\,\widehat{P}_n(\Phi_{n}-V_{\star})\,.
\end{equation}
 Applying (14.4) of  \cite[Theorem~14.0.1]{Meyn-09}) to  $\frac{g(\cdot)}{\|g\|_{V_*}}$ for any $g \in \order(V_\star)$, we have:  there exists a constant $\widetilde{m}(x)$ depending on $x$ such that
\begin{equation*}
\norm{P_*^n g(x)}^{}_{V_{\star}} \,\le\, \widetilde{m}(x) \norm{g}^{}_{V_{\star}} + \frac{\uppi_{\star}(g)}{V_{\star}(x) + 1}
\quad\,\forall x\in\XX\,,\quad\forall\,n\in\NN.
\end{equation*}
Therefore $\uppi_\star(\Phi_{n+1}) \le \uppi_\star(\Phi_{n})$ by \cref{EL4.1A}. From this it follows by induction that $\uppi_\star(\Phi_n) \le \kappa + 1 \ \forall \ n \geq 0,$ if it is so for $n=0$. Likewise, $\Phi_{n+1}-V_{\star}\ge \inf_{\XX}\,(\Phi_{n}-V_{\star})$ by
\cref{EL4.1B}. From this it follows by induction that $\phi_n \ge V_{\star} - \kappa \ \forall \ n \geq 0,$ if it is so for $n=0$.

The result then follows from these.
\end{proof}

\begin{theorem}\label{T4.1}
Assume \textup{\hyperlink{H1}{(H1)}}, and suppose $\Phi_0\in\cV(\kappa)$
for some $\kappa\in\RR$.
Then the following hold
\begin{equation}\label{ET4.1A}
\Phi_n \,\xrightarrow[n\to\infty]{}\, V_\star +
\lim_{n\to\infty}\,\uppi_\star(\Phi_n - V_{\star})
\qquad\text{in\ \ } \cL^1(\XX;\uppi_\star)\quad\text{and\ \ }\uppi_\star\text{-a.s.}\,.
\end{equation}
Also,
\begin{equation}\label{ET4.1B}
\lim_{n\to\infty}\;\babs{\nu(\Phi_n) -\uppi_\star(\Phi_n - V_{\star})}
\,=\,\beta\,.
\end{equation}
\end{theorem}

\begin{proof}
Let $\{X_n^{\star}\}_{n\in\ZZ}$ denote the stationary optimal process controlled
by $v_{\star}$.
If $\Phi_0\in\cV(\kappa)$, we have
\begin{equation*}
\sup_{n\in\NN}\,\int \abs{\Phi_{n}(x) - V_{\star}(x)}\,\uppi_\star (\D{x})\,<\,\infty
\end{equation*}
by \cref{L4.1}.
Then \cref{EL4.1A} implies that the process
\begin{equation*}
M_k\,\df\,\bigl\{ \Phi_{-{k}}(X^{\star}_{k})
- V_{\star}(X^{\star}_{k})\bigr\}_{k\le 0}
\end{equation*}
is a backward submartingale with respect to the filtration
$\{\cF_k\}_{k\le 0}\df
\bigl\{\sigma\bigl(X^{\star}_{\ell}\,,\; \ell\le k\bigr)\bigr\}_{k\le 0}$.
By Corollary V-3-13, p.\ 119, \cite{Neveu}, $M_k$ converges
a.s.\ and in the mean to some random variable $M^*$.
The latter implies the convergence in $\cL^1(\XX;\uppi_\star)$ claimed
in \cref{ET4.1A}.
By  the ergodicity of $\{X_n^{\star}\}_{n\in\ZZ}$ the $M^*$ is a constant
$\uppi_\star$-a.s. This is because $M^*$ is measurable with respect to the tail $\sigma$-field $\cap_{k\leq 0}\sigma(X^{\star}_m, m \leq k)$ which is a.s.\ trivial by the ergodicity of {\color{red} $\{X^{\star}_{-n}\}$.}

Convergence of $\Phi_n$ in $\cL^1(\XX;\uppi_\star)$, and hence
also in $\cL^1(\cB;\nu)$ by \hyperlink{A0}{(A0)}, implies that $\nu(\Phi_n)$ converges.
Since $\nu(V_\star)=\beta$ by \cref{PT3.1C}, \cref{ET4.1B} then follows from \cref{ET4.1A}.
\end{proof}

\begin{corollary}\label{C4.1}
Assume \textup{\hyperlink{H1}{(H1)}}, and suppose that $V_\star$ is bounded.
Then $\Phi_n(x)-V_\star(x)$ converges to a constant
$\uppi_\star$-a.e.\  as $n\to\infty$
for any initial condition $\Phi_0\in\fL_b(\XX)$.
\end{corollary}

\begin{proof}
This clearly follows from \cref{T4.1}, since
if $\Phi_0\in\fL_b(\XX)$, then $\Phi_0\in\cV(\kappa)$ for some $\kappa\in\RR$.
\end{proof}

\section{Relative value iteration}\label{S5}

We consider three variations of the
\emph{relative value iteration algorithm} (RVI).
All these start with initial condition $V_0\in\fL(\XX)$.

Let
\begin{equation*}
\cS f(x)\,\df\, \inf_{u\in\cU(x)}\;\bigl[\rc(x,u)+ P_u f(x)\bigr]\,,
\qquad f\in\fL(\XX)\,.
\end{equation*}
The iterates $\{V_n\}_{n\in\NN}\subset\fL(\XX)$
are defined by
\begin{equation}\label{E-RVIA}
V_{n}(x)\,=\, \cT\, V_{n-1}(x)\,\df\,
\cS V_{n-1}(x) - \nu\bigl(V_{n-1})\,,\qquad x\in\XX\,.
\end{equation}
An important variation of this is
\begin{equation}\label{E-RVIB}
\widecheck{V}_{n}(x)\,=\, \widecheck\cT\, \widecheck{V}_{n-1}(x)
\,\df\,\cS \widecheck{V}_{n-1}(x)
- \min_{\XX}\, \widecheck{V}_{n-1}\,, \qquad x\in\XX\,.
\end{equation}
Also, we can modify \cref{E-RVIA} to
\begin{equation}\label{E-RVIC}
\widetilde{V}_{n}(x)\,=\, \widetilde\cT\, \widetilde{V}_{n-1}(x)
\,\df\,\cS \widetilde{V}_{n-1}(x) - \widetilde{V}_{n-1}(\Hat{x})\,,\qquad x\in\XX\,,
\end{equation}
where $\Hat{x}\in\cB$ is some point that is kept fixed.

We let $\Hat{v}_n$ be a measurable selector from the
minimizer of \cref{E-RVIA,E-RVIB,E-RVIC} (note that all three minimizers
agree if the algorithms start with the same initial condition).
We refer to $\{\Hat{v}_n\}$
as the \emph{receding horizon} control sequence.

\begin{lemma}\label{L5.1}
Provided that $\Phi_0=V_0$, then we have
\begin{equation}\label{EL5.1A}
\Phi_n(x)-\Phi_n(y) \,=\,
V_n(x)-V_n(y)\qquad\forall\,x,y\in\XX\,,\quad\forall\,n\in\NN\,,
\end{equation}
and the same applies if $V_n$ is replaced by $\widecheck{V}_n$ or $\widetilde{V}_n$.
In addition, the convergence of $\{\Phi_n\}$
implies the convergence of $\{V_n\}$, and also that of
$\{\widecheck{V}_n\}$, $\{\widetilde{V}_n\}$ in the same space.
\end{lemma}

\begin{proof}
A straightforward calculation shows that
\begin{equation*}
V_n(x)-\Phi_n(x) \,=\, n\beta - \sum_{k=0}^{n-1} \nu(V_k)\,,
\end{equation*}
from which \eqref{EL5.1A} follows.
The proofs for $\widecheck{V}_n$ and $\widetilde{V}_n$ are completely analogous.

We have
\begin{equation*}
V_{n+1}(x)-\Phi_{n+1}(x) \,=\, V_n(x)-\Phi_n(x) +\beta - \nu(V_n)\,,
\end{equation*}
which implies that
\begin{equation*}
\nu(V_{n+1}) \,=\, \nu(\Phi_{n+1})-\nu(\Phi_n) +\beta\,.
\end{equation*}
Therefore,  convergence of $\{\Phi_n\}$ implies that
$\nu(V_{n})\to\beta$ as $n\to\infty$.
In turn, this implies the convergence of $\{V_n\}$ by \eqref{EL5.1A}.
In the case of $\{\widecheck{V}_n\}$  we obtain
$\min_{\XX}\,\widecheck{V}_n\to\beta$ as $n\to\infty$,
and analogously for $\{\widetilde{V}_n\}$.
\end{proof}

The following theorem, under
hypotheses (a)--(b), is a direct consequence of \cref{T4.1,C4.1,L5.1}.
Hypothesis \hyperlink{H2}{(H2)} is given in the beginning of the next section.

\begin{theorem}\label{T5.1}
Let one of the following assumptions be satisfied.
\begin{itemize}
\item[(a)]
\textup{\hyperlink{H1}{(H1)}} holds and $V_0\in\cV(\kappa)$ for some $\kappa\in\RR$.
Here, $\cV(\kappa)$ is as defined in \cref{E-cV}.
\item[(b)]
\textup{\hyperlink{H1}{(H1)}} holds, $V_\star$ is bounded, and $V_0\in\fL_b(\XX)$.
\item[(c)]
\textup{\hyperlink{H2}{(H2)}} holds and $V_0\in\fL(\XX)\cap\order(V_\star)$.
\end{itemize}
Then the value iteration functions in \cref{E-RVIA,E-RVIB,E-RVIC}
converge $\uppi_\star$-a.e.\ to $V_\star$ as $n\to\infty$.
In addition, if \ttup{c} holds, then
convergence is pointwise for all $x\in\XX$.
\end{theorem}

The assertions concerning \hyperlink{H2}{(H2)} are proved in the next section.
They are included in \cref{T5.1} in order to give a unified statement.

\section{Stability of the rolling horizon procedure}\label{S6}

Consider the following hypothesis.
\medskip

\begin{itemize}
\item[\hypertarget{H2}{\textbf{(H2)}}]
There exist constants $\theta_1>0$ and $\theta_2$ such that
\begin{equation*}
\min_{u\in\cU(x)}\,\rc(x,u) \,\ge\, \theta_1 V_\star(x) - \theta_2 \quad \forall x\in\XX\,.
\end{equation*}
Without loss of generality, we assume that $\theta_1\in(0,1)$.
\end{itemize}

\medskip

\begin{remark}\label{R6.1}
Hypothesis \hyperlink{H2}{(H2)} can be written in the following equivalent, but seemingly
more general form.
\medskip

\begin{itemize}
\item[\hypertarget{H2'}{\textbf{(H2$^\prime$)}}]
There exists $v\in\Usm$,
and a function $V_v\colon\XX\to[1,\infty)$
satisfying
\begin{equation*}
\min_{u\in\cU(x)}\,\rc(x,u) \,\ge\, \theta_1 V_v(x) - \theta_2 \quad \forall x\in\XX\,,
\end{equation*}
for some constants $\theta_1>0$ and $\theta_2$, and
\begin{equation*}
P_v V_v(x) - V_v(x)\,\le\, C\Ind_\cB(x) - \rc_v(x)
\end{equation*}
for some constant $C$.
\end{itemize}
\bigskip

It is clear that \hyperlink{H2}{(H2)} implies \hyperlink{H2'}{(H2$^\prime$)} (take $v =$ an optimal stationary policy),
while the converse follows
by the stochastic representation of $V_v$ and \cref{D3.2}, whereby $V_v(\cdot) \ge V_\star(\cdot)$.
From the definition of $V_\star$, it is clear by `one step analysis' that $V_\star(x) \geq \upsilon_1\min_uc(x,u) - \upsilon_2$ for suitable  constants $\upsilon_1, \upsilon_2 > 0$. Hence the above hypothesis says that $V_\star$ and $c_{min}(\cdot) := \min_uc(\cdot,u)$ have comparable growth. From the definition of $V\star$, it also follows that 
\begin{eqnarray*}
\int V_\star(y)P(dy|x,v_\star(x)) - V_\star(x) &\leq& - c_{min} (x) + C\\
&\leq& -\theta_1V_*(x) + (\theta_2 + C)
\end{eqnarray*}
for some $C > 0$. Then by Theorem 15.0.1, pp.\ 362-3, of \cite{Meyn-09}, a necessary condition for the above is that the process be geometrically ergodic under $v^*$.   
\end{remark}

Concerning the value iteration, we have the following:

\begin{theorem}\label{T6.1}
Assume \textup{\hyperlink{H2}{(H2)}}, and suppose that the
initial condition $\Phi_0$ lies in $\fL(\XX)\cap\order(V_\star)$.
Then, there exists a constant $\widehat{C}_0$ depending on $\Phi_0$ such that
\begin{equation*}
 \babs{\Phi_n(x)-V_\star(x)}\,\le\, \widehat{C}_0\bigl(1 +
\bigl(1-\theta_1)^n V_\star(x)\bigr)\qquad\forall\,x\in\XX\,,\ \forall\,n\in\NN\,.
\end{equation*}
In addition,
$\Phi_n(x)-V_\star(x)$ converges to a constant $\uppi_\star$-a.e.\ as $n\to\infty$.
\end{theorem}

\begin{proof}
Under \hyperlink{H2}{(H2)} we obtain

\begin{equation}\label{PT6.1A}
P_\star V_\star(x) \,=\, \beta - \rc_\star(x) + V_\star(x)
\,\le\, \beta + \theta_2 + (1-\theta_1) V_\star(x)\,.
\end{equation}

Let $\rho\df 1-\theta_1$, and define
\begin{equation*}
f_n(x) \,\df\, \Phi_n (x) - (1-\rho^n) \bigl(V_\star(x)
-\tfrac{\beta+\theta_2}{\theta_1}\bigr)\,.
\end{equation*}
Recall \cref{EN4.1A,ELC02}. We have
\begin{equation*}
\begin{aligned}
 f_{n+1}(x)-\widehat{P}_{n} f_{n}(x)
&\,=\, \widehat\rc_n(x)-\theta_1\rho^{n}
\bigl(V_\star(x)-\tfrac{\beta+\theta_2}{\theta_1}\bigr)
+(1-\rho^{n})(\widehat{P}_{n}-I) V_\star(x)\\[5pt]
&\,\ge\, \widehat\rc_n(x)-\theta_1\rho^{n}
\bigl(V_\star(x)-\tfrac{\beta+\theta_2}{\theta_1}\bigr)
-(1-\rho^{n})\widehat\rc_n(x)\\[5pt]
&\,=\, \rho^{n}\bigl(-\theta_1 V_\star(x)+\theta_2+\rc_n(x)\bigr)
\,\ge\, 0 \qquad \forall (x,n) \in \XX\times\NN\,,
\end{aligned}
\end{equation*}
where we also used \cref{ELC04}.
Iterating the above inequality, we get $f_n\ge \inf_\XX \Phi_0$ for all $n\in\NN$.
Assuming without loss of generality that $\Phi_0$ is nonnegative, this implies that
\begin{equation}\label{PT6.1B}
(1-\rho^n)\bigl(V_\star(x)-\tfrac{\beta+\theta_2}{\theta_1}\bigr)\,\le\,
\Phi_{n}(x) \,.
\end{equation}
On the other hand, by \cref{EL4.1A,PT6.1A}, we obtain
\begin{equation}\label{PT6.1C}
\Phi_n(x) \,\le\, V_\star(x) + \Bar{C}_0
  \bigl(\Bar{C}_1+\rho^n V_\star(x)\bigr)
\end{equation}
for some constants $\Bar{C}_0$ and $\Bar{C}_1$ which depend on $\Phi_0$.
Since $\uppi_\star(\Phi_n-V_\star)$ is bounded from above by \cref{PT6.1C},
and bounded from below by \cref{PT6.1B}, the
 result follows by the same argument as was used in
the proof of \cref{T4.1}.
\end{proof}

\begin{definition}
We say that $v\in\Usm$ is \emph{stabilizing} if
\begin{equation*}
\limsup_{N\to\infty}\;
\frac{1}{N}\;
\Exp_x^{v}\Biggl[\sum_{k=0}^{N-1}\rc_v(X_k)\Biggr]
\,<\,\infty\qquad\forall\,x\in\XX\,,
\end{equation*}
and denote the class of stabilizing controls by $\Ustab$.
\end{definition}

Recall the definition of $\Hat{v}_n$ in \cref{N4.1}.
We have the following theorem.

\begin{theorem}\label{T6.2}
Under \textup{\hyperlink{H2}{(H2)}},
for every $\Phi_0\in\fL(\XX)\cap\order(V_\star)$
there exists $N_0\in\NN$ such that the stationary
Markov control $\Hat{v}_n$ is stabilizing for any $n\ge N_0$.
\end{theorem}

\begin{proof}
Combining \cref{ELC01,PT6.1B,PT6.1C}, we obtain
\begin{equation*}
\begin{aligned}
\widehat{P}_n\,\Bigl((1-\rho^n)\bigl(V_\star(x)-\tfrac{\beta+\theta_2}{\theta_1}\bigr)\Bigr)
&\,\le\, \widehat{P}_n\,\Phi_{n}(x)\\
&\,=\, -\widehat\rc_n(x)+\Phi_{n+1}(x)\\
&\,\le\, -\widehat\rc_n(x)
+ V_\star(x) + \Bar{C}_0\bigl(\Bar{C}_1+\rho^{n+1} V_\star(x)\bigr).\,,
\end{aligned}
\end{equation*}
Rearranging, this gives
\begin{equation}\label{PT6.2A}
(1-\rho^n)\widehat{P}_n\,V_\star(x)\,\le\,
-\widehat\rc_n(x) + (1-\rho^n)\tfrac{\beta+\theta_2}{\theta_1}
+\Bar{C}_0\Bar{C}_1 +(1+\Bar{C}_0\rho)\rho^nV_\star(x)
+(1-\rho^n)V_\star(x)\,.
\end{equation}
From \hyperlink{H2}{(H2)} we have
$V_\star \le \frac{\widehat\rc_n + \beta +\theta_2}{\theta_1}$, and using
this in \cref{PT6.2A} gives
\begin{equation}\label{PT6.2B}
(1-\rho^n)\widehat{P}_n\,V_\star(x)\,\le\,
-\Bigl(1-\tfrac{1+\Bar{C}_0\rho}{\theta_1}\rho^n\Bigl) \widehat\rc_n(x)
+ \bigl(1+\Bar{C}_0\rho^{n+1}\bigr)\tfrac{\beta+\theta_2}{\theta_1}
+\Bar{C}_0\Bar{C}_1
+(1-\rho^n)V_\star(x)\,.
\end{equation}
Let $N_0\in\NN$ be large enough such that
$\tfrac{1+\Bar{C}_0\rho}{\theta_1}\rho^{N_0}<1$  and let $n\geq N_0$.
Since $0 < \rho < 1$, adding and subtracting $(1-\rho^{n-1})V_\star(x)$ on the right hand side of \cref{PT6.2B} and using it to form a telescoping sum, the fact that $V_\star$ is bounded from below in $\XX$ coupled with Fatou's lemma leads to
\begin{equation}\label{PT6.2C}
\limsup_{N\to\infty}\; \frac{1}{N}\;
\Exp_x^{\Hat{v}_n}\Biggl[\sum_{k=0}^{N-1} \rc_n(X_k)\Biggr]
\,\le\, \beta+\frac{\bigl(1+\Bar{C}_0\rho^{n+1}\bigr)(\beta+\theta_2) + \theta_1\Bar{C}_0\Bar{C}_1}
{\theta_1-\bigl(1+\Bar{C}_0\rho\bigr)\rho^n}
\qquad\forall\, n\ge N_0\,,
\end{equation}
with $\rc_n(x) \df \rc(x,\Hat{v}_n(x))$.
This shows that $\Hat{v}_n$ is stabilizing for all $n\ge N_0$.
\end{proof}

We improve the convergence result in \cref{T6.1}.

\begin{theorem}\label{T6.3}
Assume \textup{\hyperlink{H2}{(H2)}}.
Then,
for every initial condition $\Phi_0\in\fL(\XX)\cap\order(V_\star)$, the sequence
$\Phi_n(x)-V_\star(x)$ converges pointwise to a constant.
\end{theorem}

\begin{proof}
Without loss of generality, we may translate the initial condition $\Phi_0$ by
a constant so that $\Phi_n\to V_\star$ as $n\to\infty$ $\uppi_\star$-a.s.
Then of course 
\begin{equation}
\Breve\Phi^{(1)}_n\to\Breve{V}^{(1)}_\star=0 \  \  \mbox{on} \  \  \cB. \label{GoToZero}
\end{equation}
Recall also, that these functions are constant on $\cB$.
Let $\Psi_n\df\Phi_n-V_*$.
Then $\babs{\Breve\Psi^{(1)}_n}=\epsilon_n$ on $\cB$, for some sequence
$\epsilon_n\to0$.
Also
\begin{equation}\label{PT6.3A}
\babs{\Psi_n(x)} \,\le\, \widehat{C}_0\bigl(1 +
\bigl(1-\theta_1)^n V_\star(x)\bigr)\qquad\forall\,x\in\XX
\end{equation}
by \cref{T6.1},
and
\begin{equation}\label{PT6.3B}
P^n_\star V_\star(x) \,\le\, \frac{\beta+\theta_2}{\theta_1}
+ (1-\theta_1)^n V_\star(x)
\end{equation}
by \cref{PT6.1A}.
Let $\uptau_m$ denote the first exit time from the ball of radius $m$
centered at some fixed point $x_0$.
Applying the optional sampling theorem to \cref{EL4.1A}
relative to the stopping time $\uuptau\wedge n\wedge\uptau_m$,
we obtain
\begin{equation}\label{PT6.3C}
\begin{aligned}
\Breve\Psi^{(0)}_{2n}(x) &\,\le\,
\Breve\Exp^{v_\star}_{(x,0)}\Bigl[\Breve\Psi^{(1)}_{2n-\uuptau}(\Breve{X}_\uuptau)
\Ind_{\{\uuptau\le n<\uptau_m\}}
+ \Breve\Psi^{(0)}_{n}(\Breve{X}_n) \Ind_{\{n<\uuptau<\uptau_m\}}
+ \Breve\Psi^{(0)}_{2n - \uptau_m}(\Breve{X}_{\uptau_m}) \Ind_{\{\uptau_m\le\uuptau\wedge n\}}\Bigr]\\
&\,\le\, \sup_{k\ge n}\,\epsilon_k
+ \Breve\Exp^{v_\star}_{(x,0)}
\Bigl[\Breve\Psi^{(0)}_{n}(\Breve{X}_n) \Ind_{\{\uuptau> n\}} \Bigr]
+ \Breve\Exp^{v_\star}_{(x,0)}
\Bigl[\Breve\Psi^{(0)}_{2n - \uptau_m}(\Breve{X}_{\uptau_m}) \Ind_{\{\uuptau_m\le n\}} \Bigr]\,.
\end{aligned}
\end{equation}
By \cref{E-Phispl}, we have
\begin{equation}\label{PT6.3D}
\babs{\Breve\Psi^{(0)}_{n}}\,\le\, \frac{1}{1-\delta}\,\babs{\Psi_{n}}
+ \frac{\delta \epsilon_n}{1-\delta}\,.
\end{equation}
Therefore
\begin{equation*}
\limsup_{n\to\infty}\,\Breve\Exp^{v_\star}_{(x,0)}
\Bigl[\Breve\Psi^{(0)}_{n}(\Breve{X}_n) \Ind_{\{\uuptau> n\}} \Bigr]\,\le\,0
\end{equation*}
 by \cref{PT6.3A,PT6.3B,PT6.3D}
and the fact that $\Breve\Exp^{v_\star}_{(x,0)}[\uuptau]<\infty$.

By a standard application of the optional sampling theorem to the Poisson equation
 $P_\star V_\star + c_\star = V_\star + \beta$,
and keeping in mind that $V_\star$ is bounded from below,
we obtain, for some constant $\kappa_1$,
\begin{equation}\label{PT6.3E}
\Exp^{v_\star}_x\Bigl[V_\star(X_{\uptau_m})\Ind_{\{\uuptau_m\le n\}}\Bigr]
\,\le\, n\beta + \kappa_1 + V_\star(x)
\qquad\text{for all\ } m,n\in\NN\,.
\end{equation}
Therefore, combining \cref{PT6.3A,PT6.3E}, we see that there exists some constant
$\kappa_2$ such that
\begin{equation*}
\Breve\Exp^{v_\star}_{(x,0)}
\Bigl[\Breve\Psi^{(0)}_{2n - \uptau_m}(\Breve{X}_{\uptau_m}) \Ind_{\{\uuptau_m\le n\}} \Bigr]
\,\le\, \kappa_2\Bigl(\Breve\Prob^{v_\star}_{(x,0)}(\uuptau_m\le n)
+ (1-\theta_1)^n\bigl(n\beta + \kappa_1 + V_\star(x)\bigr)\Bigr)
\end{equation*}
for all $m,n\in\NN$.
This shows that the third term on the right-hand side of \cref{PT6.3C} tends to
$0$ as $m\to\infty$ for any fixed $n\in\NN$.

A slight modification of \cref{PT6.3C} also shows that
$\limsup_{n\to\infty}\Breve\Psi^{(0)}_{2n+1}(x)\le0$  for all $x\in\XX$.
Thus we have established that for each initial condition
$\Phi_0\in\fL(\XX)\cap\order(V_\star)$ there exists a constant $\Hat\kappa_0$
such that
\begin{equation}\label{PT6.3F}
\limsup_{n\to\infty}\,\Breve\Phi^{(0)}_n(x)\,\le\,\Hat\kappa_0+\Breve{V}^{(0)}_\star(x)
\quad\forall\,x\in\XX\,,\qquad\text{and\ }
 \Breve\Phi^{(1)}_n\,\xrightarrow[n\to\infty]{}\, 0\,,
\end{equation}
where we use \eqref{GoToZero}.

In the second part of the proof, we establish that
$\liminf_{n\to\infty}\,\Breve\Phi^{(0)}_n$ agrees with the value
of superior limit in \cref{PT6.3F}.
Let $\Hat{v}^n$ denote the nonstationary Markov policy
$(\Hat{v}_n,\Hat{v}_{n-1},\dotsc,\Hat{v}_{1})$ with $\Hat{v}_n$ as
defined in \cref{N4.1}.
We claim that
\begin{equation}\label{PT6.3G}
\Breve\Prob^{\Hat{v}^{2n}}_{(x,0)}[\uuptau>n] \,\xrightarrow[n\to\infty]{}\,0\,.
\end{equation}
  Summing  $(1-\delta) \ \times$ equation \eqref{one} and $\delta \ \times$ equation \eqref{three} for $x \in \cB$ and using \eqref{two} for $x \in\cB^c$,  we recover \eqref{E-VI} in view of  \eqref{E-Phispl}. To prove the claim, we apply Dynkin's formula for the stopping time $\uuptau\wedge\uptau_m\wedge n$ and Fatou's Lemma as $m\uparrow\infty$ to the value
iteration till $\uuptau\wedge\uptau_m\wedge n$ over the split chain, in order to obtain
\begin{equation}\label{PT6.3H}
\Breve\Phi_{n}(x) \,\ge\, 
\Breve\Exp^{\Hat{v}^{n}}_{(x,0)}\Biggl[\sum_{k=0}^{\uuptau\wedge n - 1}
\bigl(\Breve\rc_{\Hat{v}_{n-k}}(\Breve{X}_k)-\beta\bigr)\Biggr]
+ \Breve\Exp^{\Hat{v}^{n}}_{(x,0)}\Bigl[\Breve\Phi_{\uuptau\wedge n}\Bigr],
\end{equation}
leading to
\begin{equation}\label{PT6.3H}
\Breve\Phi_{n}^{(0)}(x) \,\ge\, 
\Breve\Exp^{\Hat{v}^{n}}_{(x,0)}\Biggl[\sum_{k=0}^{\uuptau\wedge n - 1}
\bigl(\Breve\rc_{\Hat{v}_{n-k}}(\Breve{X}_k)-\beta\bigr)\Biggr]
+ \Breve\Exp^{\Hat{v}^{n}}_{(x,0)}\Bigl[\Breve\Phi^{(1)}_{n-\uuptau}
\Ind_{\{\uuptau\le n\}}+ \Breve\Phi^{(0)}_{0}(\Breve{X}_0) \Ind_{\{\uuptau> n\}} \Bigr]
\end{equation}
for $x\in\XX$.
By  \hyperlink{A0}{\textup{(A0)}}, there exists some positive constant
$\Tilde\varepsilon_0$ such that
$\Breve\rc\bigl((x,0),u\bigr)\ge \beta +\Tilde\varepsilon_0$ 
for $(x,u)\in(\cB^c\times\Act)\cap\KK$.
Therefore,
\begin{equation*}
\sum_{k=0}^{\uuptau\wedge n - 1}
\bigl(\Breve\rc_{\Hat{v}_{n-k}}(\Breve{X}_k)-\beta\bigr)
\,\ge\, \Tilde\varepsilon_0(\uuptau\wedge n)
- (\beta+\Tilde\varepsilon_0)\sum_{k=0}^{\uuptau\wedge n - 1}
\Ind_{\cB\times\{0\}}(\Breve{X}_k).
\end{equation*}
Taking expectations, using \cref{L2.2}, we see that
\cref{PT6.3H} reduces to
\begin{equation}\label{PT6.3I}
\Breve\Phi_{n}^{(0)}(x) \,\ge\,
\Tilde\varepsilon_0\,\Breve\Exp^{\Hat{v}^{n}}_{(x,0)}[\uuptau\wedge n]
- (\beta+\Tilde\varepsilon_0)\delta_\circ
+ \Breve\Exp^{\Hat{v}^{n}}_{(x,0)}\Bigl[\Breve\Phi^{(1)}_{n-\uuptau}
\Ind_{\{\uuptau\le n\}}+ \Breve\Phi^{(0)}_{0}(\Breve{X}_0) \Ind_{\{\uuptau> n\}} \Bigr]\,.
\end{equation}
We use \cref{PT6.3F} and the hypothesis that $\Breve\Phi^{(0)}_{0}$ is bounded from below,
to obtain from \cref{PT6.3I} that
\begin{equation*}
\limsup_{n\to\infty}\,
\Breve\Exp^{\Hat{v}^{n}}_{(x,0)}[\uuptau\wedge n]
\,\le\, \kappa_3+\frac{1}{\Tilde\varepsilon_0}\,\Breve{V}^{(0)}_\star(x)
\end{equation*}
for some constant $\kappa_3$ which depends on $\Phi_0$.
This establishes \cref{PT6.3G}.

Continuing,  we assume without loss of generality
(as in the first part of the proof),
that the initial condition $\Phi_0$ is translated by a constant so that
$\Hat\kappa_0=0$ in \cref{PT6.3F}.
Recall that $\Psi_n=\Phi_n-V_*$.
Applying Dynkin's formula together with Fatou's lemma to
\cref{EL4.1B}
relative to the stopping time $\uuptau\wedge n$, we obtain
\begin{equation}\label{PT6.3J}
\begin{aligned}
\Breve\Psi^{(0)}_{2n}(x) &\,\ge\,
\Breve\Exp^{\Hat{v}^{2n}}_{(x,0)}\Bigl[\Breve\Psi^{(1)}_{2n-\uuptau}(\Breve{X}_\uuptau)
\Ind_{\{\uuptau\le n\}}+ \Breve\Psi^{(0)}_{n}(\Breve{X}_n) \Ind_{\{\uuptau> n\}} \Bigr]\\
&\,\ge\, -\sup_{k\ge n}\,\epsilon_k
+ \Breve\Exp^{\Hat{v}^{2n}}_{(x,0)}
\Bigl[\Breve\Psi^{(0)}_{n}(\Breve{X}_n) \Ind_{\{\uuptau> n\}}\Bigr]\,.
\end{aligned}
\end{equation}
Let $\widetilde{N}_0\in\NN$ be such that 
\begin{equation*}
2 \bigl(1+\Bar{C}_0\rho\bigr)\rho^{\widetilde{N}_0}/(1-\rho^{\widetilde{N}_0})
\,<\,\theta_1\,.
\end{equation*}
From \textup{\hyperlink{H2}{(H2)}}, we have,
$\widehat\rc_n\ge\theta_1 V_\star -\theta_2 - \beta$,
which, if we use in \cref{PT6.2A}, we obtain
\begin{equation}\label{PT6.3K}
\widehat{P}_n\, V_\star(x)\,\le\, \widetilde{C}_0 + \bigl(1-\tfrac{\theta_1}{2}\bigr) V_\star(x)
\qquad\forall\, n\ge \widetilde{N}_0\,,
\end{equation}
with
\begin{equation*}
\widetilde{C}_0 \,\df\,
\frac{\beta+\theta_2}{\theta_1}
+\frac{1}{1-\rho^{\widetilde{N}_0}}\bigl(\Bar{C}_0\Bar{C}_1 + \beta+\theta_2\bigr)\,,
\end{equation*}
In turn, \cref{PT6.3K} shows that
\begin{equation}\label{PT6.3L}
\Exp^{\Hat{v}^{2n}}_{(x,0)}\bigl[V_\star({X}_n)\bigr]
\,\le\, \tfrac{2\widetilde{C}_0}{\theta_1} +
\bigl(1-\tfrac{\theta_1}{2}\bigr)^n\,
V_\star(x) \qquad\forall\, n\ge \widetilde{N}_0\,.
\end{equation}
Shifting our attention to the split-chain, it is clear from
\cref{PT6.3A,PT6.3L} that
for some constant $\kappa_4$, we have
\begin{equation}\label{PT6.3M}
\Breve\Exp^{\Hat{v}^{2n}}_{(x,0)}\Bigl[
\babs{\Breve\Psi^{(0)}_{n}(\Breve{X}_n)}\Bigr]
\,\le\, \kappa_4\Bigl(1 + \bigl(1-\tfrac{\theta_1}{2}\bigr)^{2n}\,
V_\star(x)\Bigr) \qquad\forall\, n\ge \widetilde{N}_0\,.
\end{equation}
By \cref{PT6.3G,PT6.3M}, we obtain
\begin{equation*}
\liminf_{n\to\infty}\,
\Breve\Exp^{\Hat{v}^{2n}}_{(x,0)}
\Bigl[\Breve\Psi^{(0)}_{n}(\Breve{X}_n) \Ind_{\{\uuptau> n\}}\Bigr]\,\ge\,0\,,
\end{equation*}
which together with \cref{PT6.3J} shows that
$\liminf_{n\to\infty}\,\Breve\Psi^{(0)}_{2n}(x) \ge0$.
Using Dynkin's formula for $\Breve\Psi^{(0)}_{2n+1}$
in an analogous manner to \cref{PT6.3J}, we obtain the same
conclusion for this function.
Thus we have shown that
\begin{equation*}
\lim_{n\to\infty}\,\Breve\Psi^{(0)}_{n}(x)\,=\,0\,,
\end{equation*}
which completes the proof.
\end{proof}

We next show that
the sequence $\{\Hat{v}_n\}_{n\in\NN}$ is asymptotically optimal.

\begin{theorem}\label{T6.4}
In addition to \textup{\hyperlink{H2}{(H2)}}, we assume the following:
\begin{enumerate}[(a)]
\item[\ttup{a}]
The running cost $\rc$ is inf-compact on $\KK$.
\item[\ttup{b}]
There exists $\psi\in\Prob(\XX)$ such
that under the stabilizing policies $\Hat{v}_n$ in \cref{T6.2}, the controlled chain
is positive Harris recurrent and the corresponding invariant
probability measures are absolutely continuous with respect to $\psi$.
\end{enumerate}
Then for every $\Phi_0\in\fL(\XX)\cap\order(V_\star)$
the sequence $\{\Hat{v}_n\}_{n\in\NN}$ is asymptotically optimal in the sense that
\begin{equation*}
\lim_{n\to\infty}\,\widehat\uppi_n (\rc_n)\,=\,\beta
\end{equation*}
where $\widehat\uppi_n$ denotes the invariant probability measure
of the chain under the control $\Hat{v}_n$.
\end{theorem}

\begin{proof}
First, by \cref{PT6.2C},
we have
\begin{equation}\label{PT6.4A}
\widehat\uppi_n(c_n)\,\le\,
\beta+\frac{\bigl(1+\Bar{C}_0\rho^{n+1}\bigr)(\beta+\theta_2)}
{\theta_1-\bigl(1+\Bar{C}_0\rho\bigr)\rho^n}
\qquad\forall\, n\,\ge N_0\,,
\end{equation}
with $N_0$ as in the proof of \cref{T6.2}.

Combining \cref{PT6.1B}, \cref{PT6.1C}, and \textup{\hyperlink{H2}{(H2)}}, we obtain
\begin{equation}\label{PT6.4B}
\begin{aligned}
\abs{\Phi_{n+1}-\Phi_n} &\,<\,
\Bar{C}_0\Bar{C}_1 +\tfrac{\beta+\theta_2}{\theta_1}
+  (\Bar{C}_0+1) \rho^{n} V_\star(x)\\
&\,\le\,
\Bar{C}_0\Bar{C}_1 +
\theta_1^{-1}\theta_2(\Bar{C}_0+1) \rho^{n}
+\theta_1^{-1}(\Bar{C}_0+1) \rho^{n}\,\rc_n\,.
\end{aligned}
\end{equation}
Writing \cref{ELC01} as
\begin{equation}\label{PT6.4C}
\widehat{P}_n\, \Phi_n \,=\,  -\bigr(\widehat\rc_n -\Phi_{n+1}
+\Phi_{n}\bigr) + \Phi_{n}\,,
\end{equation}
and combining this with \cref{PT6.4B}, we obtain
\begin{equation}\label{PT6.4D}
\widehat{P}_n\, \Phi_n \,\le\,
\beta + \Bar{C}_0\Bar{C}_1 +
\theta_1^{-1}\theta_2(\Bar{C}_0+1)\rho^{n}
-\bigl(1-\theta_1^{-1}(\Bar{C}_0+1) \rho^{n}\bigr)\rc_n
+ \Phi_{n}\,.
\end{equation}
On the other hand, by \textup{\hyperlink{H2}{(H2)}} and \cref{PT6.1C},
we have
\begin{equation}\label{PT6.4E}
c_n \,\ge\, \theta_1\bigl(1+\Bar{C}_0\rho^n\bigr)^{-1}\bigl(\Phi_n-\Bar{C}_0\Bar{C}_1
\bigr) -\theta_2\,.
\end{equation}
Select $N_1$ such that
$2(\Bar{C}_0+1) \rho^{N_1}<\theta_1$.
Then, \cref{PT6.4D,PT6.4E} imply that there exists a constant $\Bar{C}_2$
 such that
\begin{equation}\label{PT6.4F}
\widehat{P}_n\, \Phi_n\,\le\, \Bar{C}_2 + \bigl(1-\tfrac{\theta_1}{2}\bigr)
 \Phi_n\qquad\forall\, n\ge N_1\,.
\end{equation}
Note that $\Phi_{n+1}-\Phi_n\in\order (c_n)$ by \cref{PT6.4B}.
Thus, \cref{PT6.4C,PT6.4F,PT6.4A} imply that
\begin{equation}\label{PT6.4G}
\widehat\uppi_n\bigl(\widehat\rc_n -\Phi_{n+1}+\Phi_{n}\bigr)
\,=\,0\qquad\forall\, n\,\ge N_0\vee N_1\,.
\end{equation}

 From the definition of ergodic occupation measures, it follows that they form a closed set and therefore so do the invariant probability measures under stationary strategies, which are marginals thereof. Since the latter are absolutely continuous with respect to $\psi$, it follows that their Radon-Nikodym derivatives with respect to $\psi$ are uniformly integrable and therefore weakly compact in $L^1$ by the Dunford-Pettis compactness criterion \cite[p.~27-II]{DM-78}, otherwise there would be a limit point of the invariant probability measures that is not absolutely continuous with respect to $\psi$. By the Eberlein-Smulian theorem \cite[p.~27-II]{DM-78}, this is equivalent to weak sequential compactness in $L^1$. 
Therefore every sequence of $\Lambda_n\df\frac{\D\widehat\uppi_n}{\D\psi}$
contains a subsequence which converges weakly in $L^1$.

Consider such a subsequence, which we denote as $\{\Lambda_n\}_{n\in\NN}$
for simplicity, and let $\Lambda$ be its limit.
Define $\widehat\uppi(A) \df \psi(\Ind_A\,\Lambda)$ for $A\in\Borel(\XX)$.
For every $f\in\Cc_b(\XX)$ we have
\begin{equation*}
\widehat\uppi_n (f) \,=\, \psi(f\,\Lambda_n)
\,\xrightarrow[n\to\infty]{}\, \psi(f\,\Lambda) \,=\, \widehat\uppi(f)\,.
\end{equation*}
On the other hand we have
\begin{equation}\label{PT6.4H}
\widehat\uppi_n (A) \,=\, \psi(\Ind_A\,\Lambda_n)
\,\xrightarrow[n\to\infty]{}\, \psi(\Ind_A\,\Lambda)\,,
\end{equation}
where $\widehat\uppi_n (A)=\widehat\uppi_n (\Ind_A)$ by a liberal
use of the notation.
Since, $\psi(f\,\Lambda) = \widehat\uppi(f)$ for all $f\in\Cc_b(\XX)$,
it follows of course that $\psi(\Ind_A\,\Lambda)=\widehat\uppi(A)$.
Thus $\lim_{n\to\infty}\,\widehat\uppi_n(A) =\widehat\uppi(A)$ by \cref{PT6.4H}.
Let $f_n\df\Phi_{n+1}-\Phi_{n}$.
Then
\begin{equation}\label{PT6.4I}
\widehat\uppi_n\bigl(\{\abs{f_n}>\epsilon\}\bigr)
\,=\, \int_{\{\abs{f_n}>\epsilon\}} \Lambda_n\,\D\psi
\,\le\, \sup_{m\in\NN}\,
\int_{\{\abs{f_n}>\epsilon\}} \Lambda_m\,\D\psi
\,\xrightarrow[n\to\infty]{}\,0
\end{equation}
by uniform integrability, since
$\psi\bigl(\{\abs{f_n}>\epsilon\}\bigr)\to0$ as $n\to\infty$
by \cref{T6.1}.
\Cref{PT6.4I} implies that
$f_n\to0$ in $\widehat\uppi_n$-measure in the sense of
\cite[p.~385]{Serfozo-82}.
It is also straightforward to verify using
\cref{PT6.4B,PT6.4A} and the inf-compactness of $\rc$, that
$f_n$ is tightly and uniformly $\{\widehat\uppi_n\}$-integrable in
the sense of definitions \cite[(2.4)--(2.5)]{Serfozo-82}.
Hence,
\begin{equation}\label{PT6.4J}
\widehat\uppi_n(f_n)\,\xrightarrow[n\to\infty]{}\,0
\end{equation}
by \cite[Theorem~2.8]{Serfozo-82}.
Since \cref{PT6.4J} holds over any sequence over which
$\Lambda_n$ converges in $\sigma(\cL^1,\cL^\infty)$,
it is clear that it must hold over the original sequence $\{n\}$.
The result then follows by \cref{PT6.4G,PT6.4J}.
\end{proof}

\begin{remark}
Concerning the positive Harris assumption in \cref{T6.4}, it is
clear that the Lyapunov equation \cref{PT6.4F} implies that
the controlled chain is bounded in probability.
If in addition the chain is a $\psi$-irreducible $T$-model
(see \cite[p.~177]{Tweedie-94}) then
it is positive Harris recurrent \cite[Theorem~3.4]{Tweedie-94}.
\end{remark}

The result in \cref{T6.2,T6.4} justify in particular the use of $\hat{v}_n$ for large $n$ as
a `rolling horizon' approximation of optimal long run average policy,
as is often done in Model Predictive Control, a popular approach in control
engineering practice (see, e.g., \cite{Mesbah-16}), wherein one works with time
horizons of duration $T \gg 1$ and at each time instant $t$,
the Markov control strategy optimal for the finite horizon control problem
on the horizon $[t, t+1,\dotsc, t+T]$ is used.

We present an important class of problems for which \hyperlink{H2}{(H2)} is satisfied.

\begin{example}\label{Ex6.1}
Consider a linear quadratic Gaussian (LQG) system
\begin{equation}\label{Ex6.1A}
\begin{aligned}
X_{t+1}&\,=\,A X_{t}+B U_{t}+D W_{t}\,,\quad t \ge 0 \\
X_0&\;\sim\;\cN(x_0,\Sigma_0)\,,
\end{aligned}
\end{equation}
where $X_{t}\in\Rd$ is the system state, $U_{t}\in\RR^{d_u}$ is the control,
$W_{t}\in\RR^{d_w}$ is a white noise process,
and $\cN(x,\Sigma)$ denotes the normal distribution
in $\Rd$ with mean $x$ and covariance matrix
$\Sigma$.
We assume that each $W_t\sim\cN(0,I_{d_w})$ is
i.i.d. and independent of $X_0$, and that $(A,B)$ is stabilizable.
The system is observed via a finite number of sensors scheduled or queried
by the controller at each time step.
Let $\{\gamma_{t}\}$ be a Bernoulli process indicating
if the data is lost in the network:
each observation is either received ($\gamma_{t}=1$)
or lost ($\gamma_{t}=0$).
A scheduled sensor attempts to send information to the controller through the
network; depending on the state of the network, the information may be received or lost.
The query process $\{Q_{t}\}$ takes values in
the finite set of allowable sensor queries denoted by $\QSp$.
The observation process $\{Y_t\}$  is given by
\begin{equation}\label{Ex6.1B}
Y_{t} \,=\, \gamma_t\left(C_{Q_{t-1}} X_{t}+F_{Q_{t-1}} {W}_{t}\right)\,,\quad t \ge 1,
\end{equation}
if $\gamma_t=1$, otherwise no observation is received.
The value of $\gamma_t$ is assumed to be known
to the controller at every time step.
In \cref{Ex6.1B}, $C_q$ and $F_q$ are matrices
which depend on the query $q\in\QSp$.
Their dimension is not fixed but depends on the number of sensors queried by $q$.

For each query $q\in\QSp$,
we assume that $\det(F_{q}F_{q}\transp)\neq 0$ and
(primarily to simplify the analysis) that $DF_{q}\transp=0$.
Also without loss of generality, we assume that $B$ is full rank; if not,
we restrict control actions to the row space of $B$.

The observed information is lost with a probability that depends on
the query, that is,
\begin{equation}\label{Ex6.1C}
\Prob(\gamma_{t+1}=0) = \lambda(Q_{t})\,,
\end{equation}
where the loss rate $\lambda\colon\QSp\to [0,1)$.

The running cost is the sum of a positive querying cost $\rcn\colon\QSp\to\RR$
and a quadratic plant cost $\rcp\colon\Rd\times\RR^{d_u}\to\RR$ given by
\begin{equation*}
\rcp(x,u)= x\transp R x + u\transp M u\,,
\end{equation*}
where $R,M\in\pdef$.
Here, $\pdef$ ($\psdef$) denotes the cone of real
symmetric, positive definite (positive semi-definite) $d\times d$ matrices.

The system evolves as follows.
At each time $t$, the controller takes an action ${v_{t}=(U_{t},Q_{t})}$,
and the system state evolves as in \cref{Ex6.1A}. Then the observation at $t+1$
is either lost or received, determined by \cref{Ex6.1B,Ex6.1C}.
The decision $v_{t}$ is non-anticipative, that is, it depend only on the
history $\cF_{t}$ of observations up to time $t$ defined by
\begin{equation*}
\cF_{t} \,\df\, \sigma(x_0,\Sigma_0,Y_{1},\gamma_{1},
\dotsc,Y_{t},\gamma_{t}).
\end{equation*}

This model is an extension of the one studied in \cite{Wu2008}.
More details can be found in \cite{CDC19} which considers an even
broader class of problems where the loss rate depends on the `network congestion'.

We convert the partially observed controlled Markov chain in
\cref{Ex6.1A,Ex6.1B,Ex6.1C} to an equivalent completely observed one.
Standard linear estimation theory tells us that the expected value of the state
$\widehat{X}_{t}\df\Exp[X_{t}\,|\,\cF_{t}]$ is a sufficient statistic.
Let $\widehat\Pi_t$ denote the \emph{error covariance matrix} given by
\begin{equation*}
\widehat\Pi_t = \text{cov} (X_{t}-\widehat{X}_{t})=
\Exp\bigl[(X_{t}-\widehat{X}_{t})(X_{t}-\widehat{X}_{t})\transp\bigr].
\end{equation*}
The state estimate $\widehat{X}_{t}$
can be recursively calculated via the Kalman filter
\begin{equation}\label{Ex6.1D}
\widehat{X}_{t+1} = A\widehat{X}_{t}+ B U_{t}
+ \widehat{K}_{Q_{t},\gamma_{t+1}}(\widehat\Pi_t)
\bigl(Y_{t+1} - C_{Q_{t}}(A\widehat{X}_{t}+ B U_{t})\bigr)\,,
\end{equation}
with $\widehat{X}_0=x_0$.
The Kalman gain $\widehat{K}_{q,\gamma}$ is given by
\begin{align*}
\widehat{K}_{q,\gamma}(\widehat\Pi)& \df \Xi(\widehat\Pi)\gamma C_{q}\transp
	\bigl(\gamma^{2} C_{q}\Xi(\widehat\Pi)C_{q}\transp +F_{q}F_{q}\transp\bigr)^{-1}\,, \\
\Xi(\widehat\Pi)& \df DD\transp+A\widehat\Pi A\transp,
\end{align*}
and the error covariance evolves on $\psdef$ as
\begin{equation*}
\widehat\Pi_{t+1} \,=\, \Xi(\widehat\Pi_t)
- \widehat{K}_{Q_{t},\gamma_{t+1}}(\widehat\Pi_t)C_{Q_{t}}\Xi(\widehat\Pi_t)\,,\quad
\widehat\Pi_0\,=\,\Sigma_0\,.
\end{equation*}
When an observation is lost ($\gamma_{t}=0$),
the gain $\widehat{K}_{q,\gamma_{t}}=0$ and
the observer \eqref{Ex6.1D} simply evolves without any correction factor.

Define $\cT_{q}\colon\psdef\to\psdef$ by
\begin{equation*}
\cT_{q}(\widehat\Pi) \,\df\,
\Xi(\widehat\Pi) - \widehat{K}_{q,1}(\widehat\Pi)C_{q}\Xi(\widehat\Pi)\,,
\quad q\in\QSp\,,
\end{equation*}
and an operator $\widehat\cT_{q}$ on functions
$f\colon\psdef\to\RR$,
\begin{equation*}
\widehat\cT_{q} f(\widehat\Pi) \,=\,
\bigl((1-\lambda(q)\bigr)f\bigl(\cT_{q}(\widehat\Pi)\bigr)
+ \lambda(q)f\bigl(\Xi(\widehat\Pi)\bigr)\,.
\end{equation*}
It is clear then that $\widehat\Pi_t$ forms a completely observed
controlled Markov
chain on $\psdef$, with action space $\QSp$,
and kernel $\widehat\cT_{q}$.
Admissible and Markov policies are defined as usual but with
${v_{t}=Q_{t}}$, since
the evolution of $\widehat\Pi_t$ does not depend on the state control $U_{t}$.

As shown in \cite{Wu2008}, there is a partial separation of control and observation
for the ergodic control problem which seeks to minimize the long-term average cost,
\begin{equation*}
J^{v} \,\df\, \limsup_{T\to\infty}\;\frac{1}{T}\;
\Exp^{v}\Biggl[\,\sum_{t=0}^{T-1}\bigl(\rcn(Q_{t})
+ \rcp(X_{t},U_{t})\bigr)\Biggr]\,.
\end{equation*}

The dynamic programming equation is given by
\begin{equation}\label{Ex6.1E}
V_\star(\widehat\Pi) + \varrho^*
=\min_{q\in\QSp}\;\bigl\{\rcn(s,q) + \trace(\Tilde\Pi^* \widehat\Pi)
 + \widehat\cT_{q} V_\star(\widehat\Pi) \bigr\}\,,
\end{equation}
with
$\Tilde\Pi^* \df R - \Pi^* + A\transp \Pi^* A$,
and $\Pi^*\in\pdef$ the unique solution of the algebraic Riccati equation
\begin{equation*}
\Pi^* = R + A\transp \Pi^* A {-} A\transp
\Pi^* B (M+B\transp \Pi^* B)^{-1}B\transp \Pi^* A.
\end{equation*}
If $q^*\colon\psdef\to\QSp$ is a selector of the
minimizer in \eqref{Ex6.1E}, then the policy given by
$v^*=\{U^*_t,q^*(\widehat\Pi_t\}_{t\ge0}$, with
\begin{equation*}
\begin{aligned}
U^*_{t} & \df - K^*\widehat{X}_{t}\,, \\
K^*& \df (M+B\transp \Pi^* B)^{-1}B\transp \Pi^* A\,,
\end{aligned}
\end{equation*}
and $\{\widehat{X}_t\}$ as in \eqref{Ex6.1D}, is optimal, and satisfies
\begin{equation*}
J^{v^*}=\inf_{v} J^{v}=\varrho^* + \trace(\Pi^* DD\transp)\,.
\end{equation*}
In addition, the querying component of any optimal stationary Markov policy
is an a.e. selector of the minimizer in \eqref{Ex6.1E}.

The analysis of the problem also shows that $V_\star$ is concave
and non-decreasing in $\psdef$, and thus
\begin{equation}\label{Ex6.1F}
V_\star(\Sigma)\,\le\, m_{1}^* \;\trace(\Sigma) + m_0^*\,,
\end{equation}
for some positive constants $m_{1}^*$ and $m_0^*$.
Note that the running cost corresponding to the equivalent completely observed
problem is
\begin{equation}\label{Ex6.1G}
r(q,\Sigma)\df \rcn(q)+\trace(\Tilde\Pi^* \Sigma)\,.
\end{equation}
It thus follows by \cref{Ex6.1F,Ex6.1G}
and the fact that $\Tilde\Pi^*\in\pdef$, that \hyperlink{H2}{(H2)} is satisfied
for this problem.
Note also that
the RVI, VI are given by
\begin{align*}
\varphi_{n+1}(\Sigma)&\,=\,\min_{q\in\QSp}\;
\bigl\{r(q,\Sigma) + \widehat\cT_{q}\varphi_{n}(\Sigma)\bigr\}
 - \varphi_{n}(0)\,,\\
\overline\varphi_{n+1}(\Sigma)&\,=\,\min_{q\in\QSp}\;
\bigl\{r(q,\Sigma) + \widehat\cT_{q}\overline\varphi_{n}(\Sigma)\bigr\} - \varrho^*\,,\;\;
 \overline\varphi_0=\varphi_0\,,
\end{align*}
respectively,
where both algorithms are initialized with the same function
$\varphi_0\colon\psdef\to\RR_+$.
\end{example}

\section*{Acknowledgements}
Most of this work was done during the visits of AA to
the Department of Electrical Engineering of
the  Indian Institute Technology Bombay and of VB
to the Department of Electrical and Computer Engineering at the University of Texas
at Austin, and the finishing touches were given when  both authors were at
the Institute of Mathematics of the Polish Academy of Sciences (IMPAN) in Warsaw,
for a workshop during the 2019 Simons Semester on Stochastic Modeling and Control.
The work of AA was supported in part by
the National Science Foundation through grant DMS-1715210, and in part
the Army Research Office through grant W911NF-17-1-001,
while the work of VB was supported by a J.\ C.\ Bose Fellowship. VB acknowledges
some early discussions with Prof.\ Debasish Chatterjee which spurred some of this work.

\bibliography{RVI-pseudo}

\def\cprime{$'$}
\begin{bibdiv}
\begin{biblist}

\bib{Agarwal-08}{article}{
      author={Agarwal, Mukul},
      author={Borkar, Vivek~S.},
      author={Karandikar, Abhay},
       title={Structural properties of optimal transmission policies over a
  randomly varying channel},
        date={2008},
        ISSN={0018-9286},
     journal={IEEE Trans. Automat. Control},
      volume={53},
      number={6},
       pages={1476\ndash 1491},
      review={\MR{2451236}},
}

\bib{ABG12}{book}{
      author={Arapostathis, A.},
      author={Borkar, V.~S.},
      author={Ghosh, M.~K.},
       title={Ergodic control of diffusion processes},
      series={Encyclopedia Math. Appl.},
   publisher={Cambridge University Press},
     address={Cambridge},
        date={2012},
      volume={143},
      review={\MR{2884272}},
}

\bib{RVIM}{article}{
      author={Arapostathis, Ari},
      author={Borkar, Vivek~S.},
      author={Kumar, K.~Suresh},
       title={Convergence of the relative value iteration for the ergodic
  control problem of nondegenerate diffusions under near-monotone costs},
        date={2014},
     journal={SIAM J. Control Optim.},
      volume={52},
      number={1},
       pages={1\ndash 31},
      review={\MR{3148068}},
}

\bib{Athreya-78}{article}{
      author={Athreya, K.~B.},
      author={Ney, P.},
       title={A new approach to the limit theory of recurrent {M}arkov chains},
        date={1978},
     journal={Trans. Amer. Math. Soc.},
      volume={245},
       pages={493\ndash 501},
      review={\MR{511425}},
}

\bib{BertsShreve}{book}{
      author={Bertsekas, Dimitri~P.},
      author={Shreve, Steven~E.},
       title={Stochastic optimal control: The discrete time case},
      series={Math. in Science and Engineering},
   publisher={Academic Press, Inc., New York-London},
        date={1978},
      volume={139},
        ISBN={0-12-093260-1},
      review={\MR{511544}},
}

\bib{Borkar-88}{article}{
      author={Borkar, V.~S.},
       title={A convex analytic approach to {M}arkov decision processes},
        date={1988},
     journal={Probab. Theory Related Fields},
      volume={78},
      number={4},
       pages={583\ndash 602},
      review={\MR{950347}},
}

\bib{Borkar-91}{book}{
      author={Borkar, V.~S.},
       title={Topics in controlled {M}arkov chains},
      series={Pitman Research Notes in Mathematics Series},
   publisher={Longman Scientific \& Technical, Harlow},
        date={1991},
      volume={240},
      review={\MR{1110145}},
}

\bib{Borkar-02}{collection}{
      author={Borkar, Vivek~S.},
      editor={Feinberg, E.A.},
      editor={Shwartz, A.},
       title={Convex analytic methods in {M}arkov decision processes},
      series={Internat. Ser. Oper. Res. Management Sci.},
   publisher={Kluwer Acad. Publ., Boston, MA},
        date={2002},
      volume={40},
      review={\MR{1887208}},
}

\bib{CDC19}{inproceedings}{
      author={Carroll, J.},
      author={Hmedi, Hassan},
      author={Arapostathis, A.},
       title={Optimal scheduling of multiple sensors which transmit
  measurements over a dynamic lossy network},
        date={2019},
   booktitle={Proceedings of the 58th {IEEE} {C}onference on {D}ecision and
  {C}ontrol ({N}ice, {F}rance)},
   publisher={IEEE, New York},
       pages={684\ndash 689},
}

\bib{Cavazos-96a}{article}{
      author={Cavazos-Cadena, Rolando},
       title={Value iteration in a class of communicating {M}arkov decision
  chains with the average cost criterion},
        date={1996},
     journal={SIAM J. Control Optim.},
      volume={34},
      number={6},
       pages={1848\ndash 1873},
      review={\MR{1416491}},
}

\bib{Cavazos-98}{article}{
      author={Cavazos-Cadena, Rolando},
       title={A note on the convergence rate of the value iteration scheme in
  controlled {M}arkov chains},
        date={1998},
     journal={Systems Control Lett.},
      volume={33},
      number={4},
       pages={221\ndash 230},
      review={\MR{1613070}},
}

\bib{Chatterjee-15}{article}{
      author={Chatterjee, Debasish},
      author={Lygeros, John},
       title={On stability and performance of stochastic predictive control
  techniques},
        date={2015},
        ISSN={0018-9286},
     journal={IEEE Trans. Automat. Control},
      volume={60},
      number={2},
       pages={509\ndash 514},
      review={\MR{3310179}},
}

\bib{ChenMeyn-99}{article}{
      author={Chen, Rong-Rong},
      author={Meyn, Sean},
       title={Value iteration and optimization of multiclass queueing
  networks},
        date={1999},
     journal={Queueing Systems Theory Appl.},
      volume={32},
      number={1-3},
       pages={65\ndash 97},
      review={\MR{1720550}},
}

\bib{Costa-12}{article}{
      author={Costa, O. L.~V.},
      author={Dufour, F.},
       title={Average control of {M}arkov decision processes with {F}eller
  transition probabilities and general action spaces},
        date={2012},
     journal={J. Math. Anal. Appl.},
      volume={396},
      number={1},
       pages={58\ndash 69},
      review={\MR{2956943}},
}

\bib{Vecchia-12}{article}{
      author={Della~Vecchia, Eugenio},
      author={Di~Marco, Silvia},
      author={Jean-Marie, Alain},
       title={Illustrated review of convergence conditions of the value
  iteration algorithm and the rolling horizon procedure for average-cost
  {MDP}s},
        date={2012},
     journal={Ann. Oper. Res.},
      volume={199},
       pages={193\ndash 214},
      review={\MR{2971812}},
}

\bib{DM-78}{book}{
      author={Dellacherie, Claude},
      author={Meyer, Paul-Andr\'{e}},
       title={Probabilities and potential},
      series={North-Holland Mathematics Studies},
   publisher={North-Holland Publishing Co., Amsterdam-New York},
      volume={29},
        ISBN={0-7204-0701-X},
      review={\MR{521810}},
}

\bib{Dyn-Yush}{book}{
      author={Dynkin, E.~B.},
      author={Yushkevich, A.~A.},
       title={Controlled {M}arkov processes},
      series={Grundlehren Math. Wiss.},
   publisher={Springer-Verlag, Berlin-New York},
        date={1979},
      volume={235},
        ISBN={3-540-90387-9},
      review={\MR{554083}},
}

\bib{EthKurtz}{book}{
      author={Ethier, Stewart~N.},
      author={Kurtz, Thomas~G.},
       title={Markov processes. characterization and convergence},
      series={Wiley Series in Probability and Mathematical Statistics},
   publisher={John Wiley \& Sons, Inc., New York},
        date={1986},
        ISBN={0-471-08186-8},
      review={\MR{838085}},
}

\bib{Feinberg-02}{article}{
      author={Feinberg, E.~A.},
      author={Kasyanov, P.~O.},
      author={Liang, Y.},
       title={Fatou's lemma for weakly converging measures under the uniform
  integrability condition},
        date={2019},
        ISSN={0040-361X},
     journal={Theory Probab. Appl.},
      volume={64},
      number={4},
       pages={615\ndash 630},
      review={\MR{4030821}},
}

\bib{Feinberg-17}{article}{
      author={Feinberg, E.~A.},
      author={Liang, Y.},
       title={On the optimality equation for average cost {M}arkov decision
  processes and its validity for inventory control},
        date={2017},
        ISSN={0040-361X},
     journal={Ann. Oper. Res.},
      volume={64},
      number={4},
       pages={771\ndash 790},
}

\bib{Feinberg-12}{article}{
      author={Feinberg, Eugene~A.},
      author={Kasyanov, Pavlo~O.},
      author={Zadoianchuk, Nina~V.},
       title={Average cost {M}arkov decision processes with weakly continuous
  transition probabilities},
        date={2012},
     journal={Math. Oper. Res.},
      volume={37},
      number={4},
       pages={591\ndash 607},
         url={https://doi.org/10.1287/moor.1120.0555},
      review={\MR{2997893}},
}

\bib{Feinberg-13}{article}{
      author={Feinberg, Eugene~A.},
      author={Kasyanov, Pavlo~O.},
      author={Zadoianchuk, Nina~V.},
       title={Berge's theorem for noncompact image sets},
        date={2013},
        ISSN={0022-247X},
     journal={J. Math. Anal. Appl.},
      volume={397},
      number={1},
       pages={255\ndash 259},
      review={\MR{2968988}},
}

\bib{FAM-90}{article}{
      author={Fern\'{a}ndez-Gaucherand, Emmanuel},
      author={Arapostathis, Aristotle},
      author={Marcus, Steven~I.},
       title={Remarks on the existence of solutions to the average cost
  optimality equation in {M}arkov decision processes},
        date={1990},
        ISSN={0167-6911},
     journal={Systems Control Lett.},
      volume={15},
      number={5},
       pages={425\ndash 432},
      review={\MR{1084585}},
}

\bib{FAM-92}{inproceedings}{
      author={Fern\'{a}ndez-Gaucherand, Emmanuel},
      author={Arapostathis, Aristotle},
      author={Marcus, Steven~I.},
       title={Convex stochastic control problems},
        date={1992},
   booktitle={Proceedings of the 31st {IEEE} {C}onference on {D}ecision and
  {C}ontrol, {T}ucson, {AZ}, {D}ec. 16--18},
   publisher={IEEE, New York},
       pages={2179\ndash 2180},
}

\bib{Hasminskii}{book}{
      author={Has{\cprime}minski\u{\i}, R.~Z.},
       title={Stochastic stability of differential equations},
   publisher={Sijthoff \& Noordhoff, Alphen aan den Rijn---Germantown, Md.},
     address={The Netherlands},
        date={1980},
      review={\MR{600653}},
}

\bib{HLL-90}{article}{
      author={Hern\'{a}ndez-Lerma, O.},
      author={Lasserre, J.~B.},
       title={Error bounds for rolling horizon policies in discrete-time
  {M}arkov control processes},
        date={1990},
     journal={IEEE Trans. Automat. Control},
      volume={35},
      number={10},
       pages={1118\ndash 1124},
      review={\MR{1073256}},
}

\bib{H-L-91}{article}{
      author={Hern\'{a}ndez-Lerma, On\'{e}simo},
       title={Average optimality in dynamic programming on {B}orel
  spaces---unbounded costs and controls},
        date={1991},
        ISSN={0167-6911},
     journal={Systems Control Lett.},
      volume={17},
      number={3},
       pages={237\ndash 242},
      review={\MR{1125975}},
}

\bib{ref:HLL1}{book}{
      author={Hern\'{a}ndez-Lerma, On\'{e}simo},
      author={Lasserre, Jean~Bernard},
       title={Discrete-time {M}arkov control processes. {B}asic optimality
  criteria},
      series={Appl. Math. (N. Y.)},
   publisher={Springer-Verlag, New York},
        date={1996},
      volume={30},
      review={\MR{1363487}},
}

\bib{Howard-60}{book}{
      author={Howard, Ronald~A.},
       title={Dynamic programming and {M}arkov processes},
   publisher={The Technology Press of M.I.T., Cambridge, Mass.; John Wiley \&
  Sons, Inc., New York-London},
        date={1960},
      review={\MR{0118514}},
}

\bib{Jaskiewicz-06}{article}{
      author={Ja\'{s}kiewicz, Anna},
      author={Nowak, Andrzej~S.},
       title={On the optimality equation for average cost {M}arkov control
  processes with {F}eller transition probabilities},
        date={2006},
        ISSN={0022-247X},
     journal={J. Math. Anal. Appl.},
      volume={316},
      number={2},
       pages={495\ndash 509},
      review={\MR{2206685}},
}

\bib{Mesbah-16}{article}{
      author={Mesbah, A.},
       title={Stochastic model predictive control: An overview and perspectives
  for future research},
        date={2016},
     journal={IEEE Control Systems Magazine},
      volume={36},
      number={6},
       pages={30\ndash 44},
}

\bib{Meyn-09}{book}{
      author={Meyn, Sean},
      author={Tweedie, Richard~L.},
       title={Markov chains and stochastic stability},
     edition={2nd edition},
   publisher={Cambridge University Press, Cambridge},
        date={2009},
        ISBN={978-0-521-73182-9},
      review={\MR{2509253}},
}

\bib{Meyn-97}{article}{
      author={Meyn, Sean~P.},
       title={The policy iteration algorithm for average reward {M}arkov
  decision processes with general state space},
        date={1997},
     journal={IEEE Trans. Automat. Control},
      volume={42},
      number={12},
       pages={1663\ndash 1680},
      review={\MR{1490975}},
}

\bib{Neveu}{book}{
      author={Neveu, Jacques},
       title={Mathematical foundations of the calculus of probability},
   publisher={Holden-Day, Inc., San Francisco, Calif.-London-Amsterdam},
        date={1965},
      review={\MR{0198505}},
}

\bib{Nummelin-78}{article}{
      author={Nummelin, E.},
       title={A splitting technique for {H}arris recurrent {M}arkov chains},
        date={1978},
     journal={Z. Wahrsch. Verw. Gebiete},
      volume={43},
      number={4},
       pages={309\ndash 318},
      review={\MR{0501353}},
}

\bib{Nummelin-book}{book}{
      author={Nummelin, Esa},
       title={General irreducible {M}arkov chains and nonnegative operators},
      series={Cambridge Tracts in Math.},
   publisher={Cambridge University Press, Cambridge},
        date={1984},
      volume={83},
        ISBN={0-521-25005-6},
      review={\MR{776608}},
}

\bib{Schal}{article}{
      author={Sch\"{a}l, Manfred},
       title={Average optimality in dynamic programming with general state
  space},
        date={1993},
        ISSN={0364-765X},
     journal={Math. Oper. Res.},
      volume={18},
      number={1},
       pages={163\ndash 172},
      review={\MR{1250112}},
}

\bib{Serfozo-82}{article}{
      author={Serfozo, Richard},
       title={Convergence of {L}ebesgue integrals with varying measures},
        date={1982},
        ISSN={0581-572X},
     journal={Sankhy\={a} Ser. A},
      volume={44},
      number={3},
       pages={380\ndash 402},
      review={\MR{705462}},
}

\bib{Tweedie-94}{article}{
      author={Tweedie, R.~L.},
       title={Topological conditions enabling use of {H}arris methods in
  discrete and continuous time},
        date={1994},
        ISSN={0167-8019},
     journal={Acta Appl. Math.},
      volume={34},
      number={1-2},
       pages={175\ndash 188},
      review={\MR{1273853}},
}

\bib{Vega-Amaya-03}{article}{
      author={Vega-Amaya, Oscar},
       title={The average cost optimality equation: a fixed point approach},
        date={2003},
        ISSN={1405-213X},
     journal={Bol. Soc. Mat. Mexicana (3)},
      volume={9},
      number={1},
       pages={185\ndash 195},
      review={\MR{1988598}},
}

\bib{Vega-Amaya-18}{article}{
      author={Vega-Amaya, \'{O}scar},
       title={Solutions of the average cost optimality equation for {M}arkov
  decision processes with weakly continuous kernel: the fixed-point approach
  revisited},
        date={2018},
     journal={J. Math. Anal. Appl.},
      volume={464},
      number={1},
       pages={152\ndash 163},
      review={\MR{3794081}},
}

\bib{White-63}{article}{
      author={White, D.~J.},
       title={Dynamic programming, {M}arkov chains, and the method of
  successive approximations},
        date={1963},
     journal={J. Math. Anal. Appl.},
      volume={6},
       pages={373\ndash 376},
      review={\MR{0148480}},
}

\bib{Wu2008}{article}{
      author={Wu, Wei},
      author={Arapostathis, Ari},
       title={Optimal sensor querying: general {M}arkovian and {LQG} models
  with controlled observations},
        date={2008},
     journal={IEEE Trans. Automat. Control},
      volume={53},
      number={6},
       pages={1392\ndash 1405},
      review={\MR{2451230}},
}

\bib{Yu-20}{article}{
      author={Yu, Huizhen},
       title={On the minimum pair approach for average cost {M}arkov decision
  processes with countable discrete action spaces and strictly unbounded
  costs},
        date={2020},
        ISSN={0363-0129},
     journal={SIAM J. Control Optim.},
      volume={58},
      number={2},
       pages={660\ndash 685},
      review={\MR{4074005}},
}

\bib{Yu-19}{article}{
      author={Yu, Huizhen},
       title={On linear programming for constrained and unconstrained
  average-cost {M}arkov decision processes with countable action spaces and
  strictly unbounded costs},
        date={2022},
     journal={Math. Oper. Res.},
      volume={47},
       pages={1474\ndash 1499},
}

\end{biblist}
\end{bibdiv}

\end{document}